\documentclass[12pt,a4paper]{amsart}
\usepackage[utf8]{inputenc}
\usepackage[english]{babel}
\usepackage[T1]{fontenc}
\usepackage{amsmath, dsfont}
\usepackage{mathtools}
\usepackage{amsfonts}
\usepackage{amssymb}
\usepackage{vmargin}
\usepackage{setspace}
\usepackage{mathrsfs, enumerate, csquotes, color}
\usepackage{xcolor}
\definecolor{vertfonce}{rgb}{0.20, 0.46, 0.25}
\definecolor{rougefonce}{rgb}{0.64, 0.09, 0.20}
\usepackage[breaklinks=true,
colorlinks=true,
linkcolor=rougefonce,
citecolor=vertfonce]{hyperref}

\author[S. Fournais]{S{\o}ren Fournais}
\address[S. Fournais]{Department of Mathematics,  Aarhus University,  Ny Munkegade 118,  8000 Aarhus~C, Denmark}
\email{fournais@math.au.dk}

\author[B. Helffer]{Bernard Helffer}
\address[B. Helffer]{Nantes Universit\'e,  Laboratoire Jean Leray,  Nantes, France }
\email{Bernard.Helffer@univ-nantes.fr}

\author[A. Kachmar]{Ayman Kachmar}
\address[A. Kachmar]{Lebanese University, Department of Mathematics, Nabatiye, Lebanon}
\email{akachmar@ul.edu.lb}

\author[N. Raymond]{Nicolas Raymond}
\address[N. Raymond]{Univ Angers, CNRS, LAREMA, SFR MATHSTIC, F-49000 Angers, France}
\email{nicolas.raymond@univ-angers.fr}

\title[Attractive magnetic edge]{Effective operators\\  on an attractive magnetic edge}

\makeatletter

\@addtoreset{equation}{section}
\makeatother

\theoremstyle{plain}
\newtheorem{theorem}{Theorem}[section]
\newtheorem{lemma}[theorem]{Lemma}
\newtheorem{corollary}[theorem]{Corollary}
\newtheorem{proposition}[theorem]{Proposition}

\theoremstyle{definition}
\newtheorem{remark}[theorem]{Remark}

\newtheorem{example}[theorem]{Example}

\newcommand{\R}{\mathbb{R}}
\newcommand{\C}{\mathbb{C}}
\newcommand{\N}{\mathbb{N}}
\newcommand{\Z}{\mathbb{Z}}

\renewcommand{\Re}{\mathrm{Re}}

\renewcommand{\leq}{\leqslant}	\renewcommand{\geq}{\geqslant}

\newcommand{\dd}{\mathrm{d}}

\begin{document}
	
\begin{abstract}
The semiclassical Laplacian with discontinuous magnetic field is considered in two dimensions. The magnetic field is sign changing with exactly two distinct values and is discontinuous along a smooth closed curve, thereby producing an attractive magnetic edge.  Various accurate spectral  asymptotics are established by means of a   dimensional reduction involving a microlocal phase space  localization  allowing to deal with the discontinuity of the field.
\medskip
\end{abstract}
	
	\maketitle
	
\section{Introduction}\label{sec:intro}

\subsection{General framework}
In this article, we consider the magnetic Laplacian on the plane $\mathbb R^2$,
  \begin{equation}\label{eq:P}
\mathcal P^a_{h}:=(-ih\nabla+\mathbf A)^2=\sum_{j=1}^2 (-ih \partial_{x_j} +A_j)^2,
\end{equation}
with magnetic potential $\mathbf A:=(A_1,A_2)\in H^1_\textrm{loc}(\mathbb R^2;\mathbb R^2)$, generating  the piecewise constant magnetic field
\begin{equation}\label{eq:B-ms}
B=\mathbf 1_{\Omega_1}+a\mathbf 1_{\Omega_2}\,,
\end{equation}
where $-1\leq  a\leq a_0$ and  $a_0$ is a fixed negative constant. Here   $h>0$ is a small parameter (the semiclassical parameter). Throughout this paper, we assume that
\begin{equation}\label{eq:Gam}
\left\{\,\begin{aligned}
\Omega_1\subset\mathbb R^2\text{ is a }&\text{connected and simply connected open set,}~\Omega_2=\mathbb R^2\setminus\overline{\Omega}_1,\\
&\Gamma:=\partial\Omega_1 \text{ is a }C^\infty\text{ smooth closed curve.}\\
\end{aligned}\,\right\}
\end{equation}
and we refer to $\Gamma$ as the magnetic edge (see Fig~\ref{fig1}). We will denote the length of $\Gamma$ by $|\Gamma|=2L$.
\begin{figure*}[ht!]
\includegraphics[width=8cm]{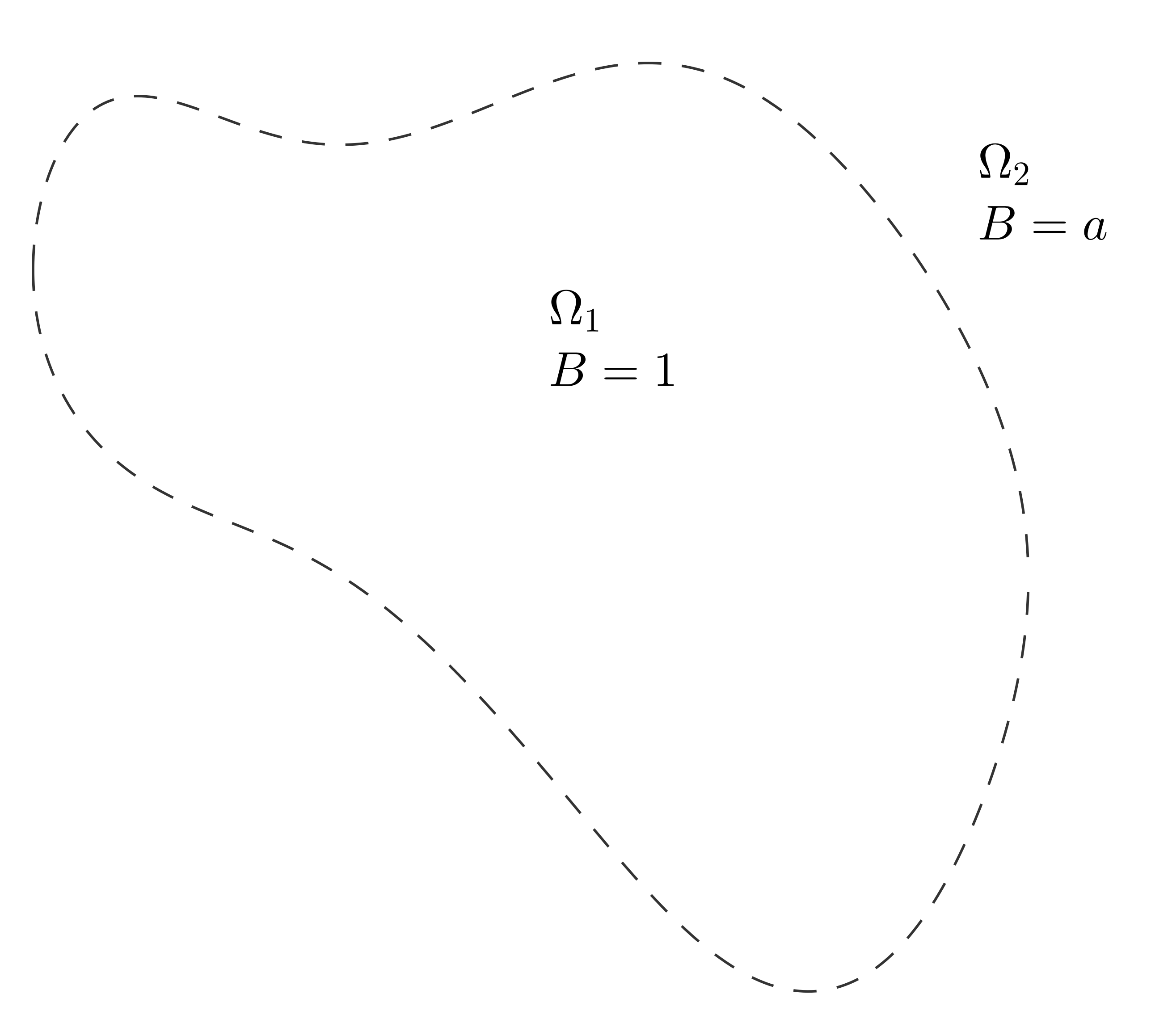}
\caption{The plane  $\mathbb R^2=\Omega_1\cup\Omega_2\cup\Gamma$ with  the  edge  $\Gamma=\partial\Omega_1$ dashed.}\label{fig1}
\end{figure*}

The operator  $\mathcal P^a_h$ is  self-adjoint in $L^2(\mathbb  R^2)$ with domain 
\begin{equation}\label{eq:DomP}
\textsf{ Dom}(\mathcal P^a_{h})=\{u\in L^2(\mathbb R^2)~:~(-ih\nabla+\mathbf A)^{j}u\in  L^2(\mathbb R^2),~j=1,2\}.
\end{equation}
Its  essential spectrum   is determined by the magnetic field at infinity (in our case it is equal to $a$). More precisely,  by Persson's lemma, we have  
\[\inf\mathrm{sp}_\textrm{ess}(\mathcal P^a_h)\geq |a| h\,.\]
The purpose of this paper is to study the spectrum of $\mathcal P^a_h$ in  the energy window
$J_h=[0, Eh]$ with $E\in(0,|a|)$ a fixed constant (thus, we analyse the spectrum below the essential spectrum) and in the semiclassical limit $h\to 0$. 
We denote by $\lambda_n(\mathcal P^a_h)$ the $n$'th eigenvalue of $\mathcal P^a_h$, and have
\begin{align}
\mathrm{sp}(\mathcal P^a_h) \cap [0, Eh] = \{ \lambda_n(\mathcal P^a_h) \}_{n=1}^N,
\end{align}
with $N=N(h)$.

Let us stress that our spectral analysis will be uniform with respect to $a\in[-1,a_0]$ and that the condition on the sign of $a$ is crucial since we will see that it implies a localization of the eigenfunctions associated with eigenvalues in $J_h$ near the edge $\Gamma$.  That is why we will say that the edge is attractive.  

\subsection{Heuristics, earlier results, and motivation}
\subsubsection{Analogy with an electric well and mini-wells}
The  problem investigated in this paper shares common features  with the semiclassical asymptotics of the Schr\"{o}dinger operator,  $-h^2\Delta+V$,  with an electric potential $V$,  in the full plane, see \cite{HeSj,HeSj5,Simon, HS6}.  In this context,  the \enquote{well} is  the set $\Gamma_V:=\{x\in\mathbb R^2\,|\,V(x)=\min_{\mathbb R^2} V\}$, which  attracts the bound states in  the limit $h\to0$.  The well is said to be \emph{non-degenerate} if  $\Gamma_V$ is a regular manifold, in which case the bound states might be localized near some points of $\Gamma_V$,  the \emph{mini-wells}. This phenomenon of   mini-wells is a manifestation of a multi-scale localization of the bound states.  Interestingly,  this phenomenon occurs also in the setting of the magnetic Laplacian,  with a Neumann boundary condition, or with a magnetic field having a step-discontinuity as in the present article.  In  particular,  if we consider the Neumann Laplacian with a constant magnetic field in a bounded, smooth domain,  the  boundary  of the domain  acts as the \enquote{well}
and the set of points of the boundary  with maximum curvature  acts as the  \enquote{mini-well} (see \cite{HM01,FH06}).

\subsubsection{Some known results}
Recently in \cite{AHK, AK20},  the operator  $\mathcal  P^a_h$ was considered in $L^2(\Omega)$ with  Dirichlet boundary condition on $\partial \Omega$,  $\Omega_1\subset \Omega$ and $\Gamma$  a  smooth curve that meets $\partial\Omega$ transversely.   The edge $\Gamma$ acts  as  the ``well'' and the  set of  points  of $\Gamma$  with maximum curvature acts  as  the ``mini-well''. 
Moreover, when the curvature has a unique non-degenerate maximum along the edge $\Gamma$,  an  accurate eigenvalue asymptotics displaying the splitting of the individual eigenvalues of $\mathcal P^a_h$ has been derived in \cite[Thm.~1.2]{AHK},   when $-1<a<0$. This result is clearly reminiscent of \cite{FH06}.

\subsubsection{Motivation}
In the present article, we propose another perspective on the problem. Our spectral analysis will be uniform in various ways. Firstly, it will allow to derive, given some $E\in(0,|a|)$,  an effective operator in the whole energy window $J_h=[0,Eh]$  with $h\in (0,h_0]$. In particular, the same strategy will provide us with Weyl estimates (estimating the number of eigenvalues in $J_h$) and the behavior of the individual eigenvalues. Secondly, it will also be uniform with respect to the parameter $a\in[-1,a_0]$. This uniformity is the key to the understanding of the transition between the regimes $a\in(-1,0)$ and $a=-1$. This is all the more motivating since the mini-well phenomenon does not occur when $a=-1$. It is indeed rather satisfactory to have a point of view encompassing quite different phenomena and showing their unity.

\subsection{The band functions}
The statement of our main results involves  a family  of 1D Schr$\ddot{\rm o}$dinger operators and their lowest eigenvalues, namely
the operators obtained when the magnetic step is along a straight line, in which case a dimensional reduction is possible.
This family has been the object of recent works (see \cite{AK20, HPRS16}). Let us briefly recall some of its basic properties.
Straightening the edge $\Gamma$ locally,   it is natural to consider the following \enquote{tangent} operator on $\R^2$ with magnetic field
\[B=\textrm{ curl\,}\mathbf A=\mathbf 1_{\mathbb R_+\times\mathbb R }+a\mathbf 1_{\mathbb R_-\times\mathbb R}\,,\]
where $ a\in[-1,a_0]$ is a fixed constant\footnote{Our investigation concerns the attractive magnetic edge,  which is the case when $a<0$.  In the opposite case,  $a \in (0,1)$,  the magnetic edge will no longer attract the bound states,  since  $\mu_a(\sigma)$ (defined in \eqref{eq:defmua}) becomes a monotone decreasing function with $\inf_{\sigma \in{\mathbb R}} \mu_a(\sigma) = a$.}. This operator is explicitely given by
\begin{equation}\label{eq.Ptgt}
\mathcal{P}^{\mathrm{tgt}}_h=h^2D_t^2+(hD_s-tb_a(t))^2\,,\quad b_a(t)=\mathbf{1}_{\mathbb R_+}(t)+a\mathbf{1}_{\mathbb R_-}(t)\,.
\end{equation}
By using a rescaling and a partial Fourier transformation along the straight edge $t=0$,  we are led to consider the analytic family of Schr\"{o}dinger operators
\begin{equation}\label{eq:ha}
	\mathfrak h_a[\sigma]=-\partial^2_t+\big(\sigma-b_a(t)t\big)^2\,,
\end{equation}
with domain
\[ B^2(\R)=\{u\in  L^2(\R) : u''\in L^2(\R)\,,t^2u\in L^2(\R)\}\,,\]
where $\sigma\in\mathbb R$ is a parameter. \\
The operator $\mathfrak h_a[\sigma]$ is self-adjoint in $L^2(\R)$ and has compact resolvent. We denote by $(\mu^{[n]}_{a}(\sigma))_{n\geq 1}$ the non-decreasing sequence of the eigenvalues (repeated according to their multiplicity) of $\mathfrak h_a[\sigma]$. For shortness, we let 
\begin{align}\label{eq:defmua}
\mu_a(\sigma)=\mu^{[1]}_{a}(\sigma) = \inf \mathrm{sp}(\mathfrak h_a[\sigma])\,.
\end{align}

By the Sturm-Liouville theory, we have the following proposition. 

\begin{proposition}\label{lem.simple}
	All the eigenvalues of $\mathfrak h_a[\sigma]$ are simple.	
	The eigenfunction associated with $\mu^{[n]}_{a}(\sigma)$ has exactly $n-1$ simple zeroes on $\R$.
\end{proposition}

The functions $\mu^{[n]}_{a}(\sigma)$,  are called the band functions.  When $a=1$,  we are reduced to the harmonic oscillator and $\mu^{[n]}_a(\sigma)=2n-1$.  When $-1\leq a<1$,  the functions $\mu^{[n]}_a(\sigma)$ are no more constant functions, see \cite{HPRS16}.  The lowest band function,   $\mu_{a}(\sigma)$ is studied in \cite{AK20}.

\begin{proposition}[\cite{AK20, HPRS16}]\label{prop.mu-a}
	For all $n\geq 1$,  the function $\mu^{[n]}_{a}$ is analytic as a function of $\sigma$.   Moreover,  the lowest band function satisfies
	\begin{equation}\label{eq:limit-mu-a}
		\lim\limits_{\sigma\to-\infty}\mu_a(\sigma)=+\infty,\quad  \lim\limits_{\sigma\to+\infty}\mu_a(\sigma)=|a|\,,
	\end{equation}
	and $\mu_a$ has a unique critical point, which is a non-degenerate minimum $\beta_a\in(0,|a|)$, attained at $\sigma(a)>0$.
\end{proposition}
In light of Proposition~\ref{prop.mu-a},  we write,  for $E\in(0,|a|)$,
\begin{equation}\label{eq:xi-pm}
	\mu^{-1}_a([\beta_a,E])=\big[\sigma_-(a,E),\sigma_+(a,E)\big]\,,
\end{equation}
where $-\infty<\sigma_-(a,E)<\sigma(a)<\sigma_+(a,E)<+\infty$.

\subsection{Main results}
Our analysis will reveal that the semiclassical spectral asymptotics of $\mathcal{P}^a_h$ in the interval $[0,Eh]$ is governed by that of an effective operator acting on the edge $\Gamma$.   In particular,  we obtain accurate asymptotics for the low-lying eigenvalues of $\mathcal P^a_h$ highlighting  a significant difference between the cases where $-1<a<0$ and $a=-1$.

\begin{theorem}[Case $-1<a<0$]\label{corol.a>-1}
Assume that $k$ has a unique maximum,  which is non-degenerate:
\[ k_{max}:=\max_\Gamma k=k(s_{\max})\,,\quad k''(s_{\max})<0\,.\]
For all $a\in(-1,0)$,  there exists $C(a)>0$ such  that,  for all $n\geq 1$,
\[\lambda_n(\mathcal{P}^a_h)=
\beta_a h-C(a)k_{\max}h^{\frac32}+(n-\frac 12)h^{\frac74}\sqrt{-C(a)\mu''_a(\sigma(a))k''(s_{\max})}+o_n(h^{\frac74})\,.
\]	
\end{theorem}

\begin{remark}
~
\begin{enumerate}[\rm i)]
\item Theorem~\ref{corol.a>-1}  recovers the  asymptotics obtained in \cite{AHK}.
The constant is given by $C(a)= -M_3(a)>0$, with $M_3(a)$ defined in \eqref{eq:moments} and calculated in \eqref{eq:m3*}.
\item The  asymptotics in Theorem~\ref{corol.a>-1}    is consistent with the phenomenon observed in surface superconductivity (see \cite{FH06} and references therein) and   the semiclassical analysis for the Schr\"odinger operator with a degenerate well in \cite{HeSj5}.  In this comparison,  the well corresponds here to $\Gamma$ and the mini-wells correspond to the points of maximal curvature. 
\item Actually, the proof of Theorem~\ref{corol.a>-1} provides us with a uniform description of the spectrum in $[0,Eh]$ and could also help determining the behavior of the eigenvalues close to $Eh$ when $E$ is non-critical for $\mu_a$, \emph{i.e.}, when $E\neq \beta_a$. In the context of the Robin Laplacian, such considerations are the object of the ongoing work \cite{FLTRVN22}. Note also that there are some results 
high up in the spectrum in the recent work \cite{GV},  where Dirichlet conditions are considered.
\item It might happen that $k$ does not have a unique minimum and even that $\Gamma$ has some symmetry properties. In this case,  tunneling occurs and the eigenvalue splitting is  exponentially small (see \cite{FHK22}).   The proof is similar   to  the case of the Laplacian  with a  constant magnetic field and Neumann boundary condition in a symmetric domain \cite{BHR22}.
 \end{enumerate}
\end{remark}
When $a=-1$, we will prove that $C(a)=0$ and thus the second and third terms in  the  asymptotics formally vanish.   We still get  accurate estimates for the low-lying eigenvalues
of  $\mathcal P^a_h$ when $a=-1$,  which  involves an operator on the edge $\Gamma\simeq [-L,L)$, whose half-length is denoted by $L$.
\begin{theorem}[Case $a=-1$]\label{corol.a=-1}  There exists   $C_0<0$ such that,  for every $n\in\N$,  we have as $h\to0$,
\[\lambda_n(\mathcal{P}_h^{ \{a=-1\}})=
\beta_{-1} h+h^2\gamma_n(h)+o_n(h^2)\,,\]
 where $\gamma_n(h)$ is the non-decreasing sequence of the eigenvalues of the differential operator
\[\frac{\mu''_{-1}(\sigma(-1))}{2}\left(D_s+\alpha_h\right)^2+C_0k(s)^2\,,\quad\mbox{ with } D_s=-i\partial_s\,,\]	
acting on $[-L,L)$ with periodic boundary conditions,  and
\begin{equation}\label{defalphah}
\alpha_h:=  \frac{|\Omega_1|}{2Lh}  - \frac{\sigma(-1)}{\sqrt{h}}\,.
\end{equation}
Here $|\Omega_1|$ is the area  of  $\Omega_1$.
\end{theorem}
 The  quantity $\alpha_h$ in \eqref{defalphah} involves the circulation  of the magnetic potential along $\Gamma$. In  fact,   by  Stokes' Theorem,  the circulation satisfies 
\[\frac1{|\Gamma|}\int_\Gamma \mathbf A\cdot \tau \,\dd s(x)=\frac1{|\Gamma|}\int_{\Omega_1}  {\rm curl\,}\mathbf A \,\dd x= \frac{|\Omega_1|}{2L}.\]

 At the first glance,  Theorems~\ref{corol.a>-1} and \ref{corol.a=-1} seem independent. However,  they both  result  from  the analysis of  the effective operator of  $\mathcal P_h^a$ (see Theorem~\ref{thm.main} below),  which provides  us with  an accurate spectral description for $-1\leq a<0$.    

This effective operator can be described as an $\hbar$-pseudodifferential operator on  $\R$ with a $2L$-periodic symbol with respect to the space variable, and acting on $2L$-periodic functions.    Here and along the whole paper the parameter
\begin{equation}
\hbar:=h^{1/2}
\end{equation}
is called the  \emph{effective} semiclassical parameter.  Let us describe the shape of our effective operator.  For a given symbol  $p_\hbar(s,\sigma)\in S_{\R^2}(1)$\,\footnote{that is a smooth bounded function on $\R^2$ such that its derivatives at any order are also bounded, uniformly in $\hbar\in(0,1]$.},  we consider  the Weyl quantization, \emph{i.e.}, the operator defined by 
\begin{equation}\label{eq.ph-w-intro}
({\rm Op}_\hbar^w(p_\hbar) u)(s)=\frac{1}{2\pi\hbar} \int_{\R^2} e^{i(s-\tilde s)\cdot \sigma/\hbar} p_\hbar \left(\frac{s+\tilde s}{2}, \sigma\right)u(\tilde s)   \mathrm{d}\tilde s \mathrm{d}\sigma\,.
\end{equation}
For an introduction to pseudo-differential operators, the reader is referred for instance to \cite{Zworski}, where rigorous definitions are given and several fundamental properties are established. These operators being well defined on $\mathcal S(\mathbb R)$, they can be extended by duality as operators on $\mathcal S'(\mathbb R)$. We now underline that, if $p_{\hbar}(s+2L,\sigma)=p_{\hbar}(s,\sigma)$, then ${\rm Op}_\hbar^w(p_\hbar)$ transforms all the $2L$-periodic distributions into $2L$-periodic distributions.  In fact,  ${\rm Op}_\hbar^w(p_\hbar)$ also preserves the space of $2L$-periodic functions that are in $L^2_{\rm{loc}}$, denoted by $L^2_{2L}(\R)$  (see Section~\ref{s-sec.p}).

Such an induced operator will give us our effective operator and we will call it \emph{a pseudodifferential operator on the edge}, $s$ representing the coordinate on $\Gamma$ (parametrised by arc-length).   

 The  main result in this article is the  following.
\begin{theorem}[Spectral reduction to the edge]\label{thm.main}~\\ 
There exists a self-adjoint $\hbar$-pseudodifferential operator (with symbol $p_\hbar^{\mathrm{eff}}\in S_{\R^2}(1)$) on the edge,  whose principal symbol coincides with $\mu_a$ below $E$,  such that the spectrum of $\mathrm{Op}^w_\hbar (p^{\mathrm{eff}}_\hbar)$ is discrete in $[0,E]$ for $\hbar$ in  some interval $(0,\hbar_0]$.  
 
Moreover,  for all  $n\in\N$ such that $\lambda_n(\mathcal{P}_h^a)\in J_h=[0,Eh]$,  we have as $h\to0$,
\[\lambda_n(\mathcal{P}^a_h)= h\lambda_n(\mathrm{Op}^w_\hbar (p^{\mathrm{eff}}_\hbar))+o(h^{2})\,,\]
uniformly with respect to $a\in[-1,a_0]$, where $-1<a_0<0$.  Here $\lambda_n(\mathrm{Op}^w_\hbar (p^{\mathrm{eff}}_\hbar))$ denotes the $n$-th eigenvalue of $\mathrm{Op}^w_\hbar (p^{\mathrm{eff}}_\hbar)$.
\end{theorem}
The discreteness of the spectrum of such an $\hbar$-pseudodifferential operator,   for  $\hbar$ small enough, is rather classical.  Indeed,   fixing $E^+\in(E,|a|)$,  we shall see that the principal symbol of $p_\hbar^{\mathrm{eff}}$ coincides with $\mu_a$ below $E^+$ and thus, since $\mu_a$ has a unique minimum,  we can consider a smooth function of $\sigma$ with compact support, denoted by $\chi$, such that $p_\hbar^{\mathrm{eff}}(s,\sigma)+\chi(\sigma)\geq E^+$. Since $\mathrm{Op}^w_\hbar \chi$ is a compact operator on $L^2_{2L}(\R)$, we get that the essential spectra of $\mathrm{Op}^w_\hbar (p^{\mathrm{eff}}_\hbar)+\mathrm{Op}^w_\hbar \chi$ and $\mathrm{Op}^w_\hbar (p^{\mathrm{eff}}_\hbar)$ coincide. By using the G\aa rding inequality, this essential spectrum is contained in $ (E,+\infty)$.

The power of Theorem~\ref{thm.main} is  that  it yields the two different asymptotics in Theorems~\ref{corol.a>-1}~and~\ref{corol.a=-1}.  The analysis in  \cite{AHK} only works for $-1<a<0$,  in which case the eigenfunctions are localized near the edge point(s) of maximal curvature,  while in the perfectly symmetric situation when $a=-1$, the localization near the edge is displayed via an effective operator essentially  independent of $h$ (and thus the corresponding eigenfunctions are not particularly localized near specific points on the edge, even in the limit $h\to 0$).

Of course,   the present statement of Theorem \ref{thm.main} is not very informative if we do not describe the effective operator (see \eqref{eq.exp.eff-symb} for  the expression of $p_\hbar^{\mathrm{eff}}$, involving the curvature $k$ along the edge $\Gamma$,  viewed as a function of the arc-length $s$).  However,  it already gives an idea of the dimensional reduction approach  using the tools developed in \cite{Keraval} and inspired by \cite{GMS91, M07}.

Besides the accurate asymptotics  of the low-lying eigenvalues  obtained  in Theorems \ref{corol.a>-1} and \ref{corol.a=-1},  another interesting result that follows from Theorem~\ref{thm.main}    is a Weyl estimate.

\begin{theorem}[Asymptotic number  of edge states]\label{corol.weyl}
We have
\[N(\mathcal{P}^a_h, Eh)\underset{\hbar\to 0}\sim \frac{L(\sigma_+(a,E)-\sigma_-(a,E))}{\pi\sqrt{h}}\,.\]
\end{theorem}
The above Weyl estimate is similar to the one for the  Neumann Laplacian with a magnetic field obtained by  purely variational methods 
not involving pseudodifferential techniques in \cite{F07,  FK,  KK}.  
\begin{remark}
\begin{enumerate}[\rm i)]
\item Our work does not cover the case when $\Gamma$ has corners, in which case a strategy of dimensional reduction might be inefficient (as in the case for the Neumann magnetic Laplacian on corner domains, see \cite{BN, BND06}). 
\item Another interesting question is  to analyze the behavior of  the spectrum   near the Landau level $|a|h$, where we loose the uniformity in our estimates and we can expect that another regime occurs.
\end{enumerate}
\end{remark}

\subsection{Organization}

In Section~\ref{sec:FlatEdge}, we discuss and recall some elementary properties of the model in $\R^2$ with a flat edge. 
Section~\ref{sec:red} is devoted to the description of the Frenet coordinates along the  edge $\Gamma$ and the reduction of our problem to the study of an  operator in a neighborhood of $\Gamma$.
In Section~\ref{sec:pseudo}, we express the operator obtained in Section~\ref{sec:red}  as an $\hbar$-pseudodifferential operator with operator symbol and expand this operator in powers of $\hbar$. In Section~\ref{sec:Gr}, we use a Grushin problem to construct a parametrix (that is an approximate inverse) for the operator introduced in Section~\ref{sec:red}.
In Section~\ref{sec:proofs}, we deduce accurate eigenvalue estimates from the Grushin reduction, finish the proof of Theorem~\ref{thm.main}, and show how it yields the other theorems announced in the introduction.

\section{The flat edge  model}\label{sec:FlatEdge}
This section is devoted to the study of the flat edge model \eqref{eq.Ptgt} and more precisely to the properties of the fibered family \eqref{eq:ha}. We recall that our analysis holds for $[-1,a_0]$, with $-1<a_0<0$.

\subsection{More on the band functions}\label{sec.chi1}
We will use the following lemma.
\begin{lemma}\label{lem.mu2}
For all $\sigma\in\R$, we have
\[\mu_a^{[2]}(\sigma)> |a|\,.\]	
\end{lemma}
\begin{proof}
Let us consider the $L^2$-normalized eigenfunction $u:=u^{[2]}_{a,\sigma}$ associated with $\mu_a^{[2]}(\sigma)$. We have
\begin{equation*}
-u''(t)+(\sigma-tb_a(t))^2u(t)=\mu_a^{[2]}(\sigma) u(t)\,.
\end{equation*}
By the Sturm-Liouville theory, $u$ has exactly one simple zero $t_0$. Assume first that $t_0\geq 0$. Then, for all $t\geq 0$,
\[-u''(t+t_0)+(\sigma-(t+t_0))^2u(t+t_0)=\mu_a^{[2]}(\sigma) u(t+t_0)\,.\]
The (non-zero) function $v=u(\cdot+t_0)$ is an eigenfunction of the Dirichlet realization on $\R_+$ of $-\partial_t^2+(\sigma-t_0-t)^2$. Since $v$ does not vanish on $\R_+$, we have $\mu_a^{[2]}(\sigma)=\mu^{\mathrm{Dir}}_1(\sigma-t_0)>1\geq |a|$.
Now, assume that $t_0<0$. Then, for all $t\leq 0$,
\[-u''(t+t_0)+(\sigma-a(t+t_0))^2u(t+t_0)=\mu_a^{[2]}(\sigma) u(t+t_0)\,.\]
In the same way, we infer that $\mu_a^{[2]}(\sigma)>|a|$.
\end{proof}

  For later use,  we can consider a smooth bounded increasing function $\chi_1$ on $\R$ such that $\chi_1(\sigma)=\sigma$ on a neighborhood of the interval $[\sigma_-(a,E^+),\sigma_+(a,E^+)\big]$, see \eqref{eq:xi-pm}. In particular $\mu_a\circ\chi_1$ has still a unique minimum at $\sigma_a$, which is not degenerate and not attained at infinity (since $\liminf_{|\sigma| \rightarrow +\infty} \mu_a (\chi_1(\sigma) )> \mu_a(\sigma_a)$).
  The functions $\overset{\circ}{\mu}^{[n]}_a=\mu^{[n]}_a\circ\chi_1$ will serve as bounded versions of $\mu^{[n]}_a$.  We denote by $\overset{\circ}{u}_\sigma$ the positive and normalized ground state of 
\begin{equation}\label{eq:circ-n}
\overset{\circ}{n}_0(\sigma):=\mathfrak  h_a[\chi_1(\sigma)]\,,
\end{equation}
where $\mathfrak h_a$ is defined in \eqref{eq:ha}.

We can express the  projection on ${\rm  span}(\overset{\circ}{u}_\sigma)$ as  $ \Pi^*(\sigma)\Pi(\sigma)$ where
\[ \Pi(\sigma)=\langle\cdot,\overset{\circ}{u}_{\sigma}\rangle\quad{\rm and}\quad \Pi^*(\sigma)=\cdot\, \overset{\circ}{u}_{\sigma}\,, \]
where we underline that $\Pi^*(\sigma)\in\mathcal L (\mathbb C, L^2(\mathbb R) )$.
Thanks to Lemma \ref{lem.mu2} (and the spectral theorem), for all $z\in [0,E^+]$, we can consider the regularized resolvent\footnote{ Since $[0,|a|)\cap{\rm sp}(\overset{\circ}{n}_0(\sigma))=\{\mu_a^{[1]}(\sigma)\}$,   $\overset{\circ}{n}_0(\sigma)-z$ can be inverted on the orthogonal complement of $\overset{\circ}{u}_\sigma$, for $0\leq z<E_+<|a|$.}
\[\overset{\circ}{\mathfrak R}_{0,z}(\sigma)=   (\overset{\circ}{n}_0(\sigma)-z)^{-1}\,\big(\mathrm{Id}-\Pi^*(\sigma)\Pi(\sigma)\big)\,.\]

\begin{example}
As mentioned in the introduction, we will work with pseudodifferential operators in the $s$-variable (parallel to the boundary).
A key example is given by $\Pi$ above. We view $(s,\sigma) \mapsto \Pi(\sigma) \in \mathscr{L}(L^2(\R),\C)$ as an operator-valued symbol.
Thereby we get, using the Weyl quantization of \eqref{eq.ph-w-intro} in the introduction, for $\varphi=\varphi(s,t)$,
\begin{align*}
({\rm Op}_\hbar^w(\Pi) \varphi)(s)&=\frac{1}{2\pi\hbar} \int_{\R^2} e^{i(s-\tilde s)\cdot \sigma/\hbar} (\Pi(\sigma) \varphi)(\tilde s)   \mathrm{d}\tilde s \mathrm{d}\sigma \\
&= \frac{1}{2\pi\hbar} \int_{\R^2} e^{i(s-\tilde s)\cdot \sigma/\hbar} \int_{\R} \varphi(\tilde s,t)  \overline{\overset{\circ}{u}_\sigma(t)} \mathrm{d}t  \mathrm{d}\tilde s \mathrm{d}\sigma.
\end{align*}
Similarly, $(s,\sigma) \mapsto \Pi(\sigma)^{*} \in \mathscr{L}(\C,L^2(\R))$ is an operator-valued symbol and for $\psi=\psi(s)$, we have
\begin{align*}
({\rm Op}_\hbar^w(\Pi^{*}) \psi)(s,t)
&= \frac{1}{2\pi\hbar} \int_{\R^2} e^{i(s-\tilde s)\cdot \sigma/\hbar}\psi(\tilde s) \overset{\circ}{u}_\sigma(t) \mathrm{d}\tilde s \mathrm{d}\sigma.
\end{align*}

\end{example}

\begin{proposition}\label{prop.Grushin0}
For all $ z\in [0,E^+]$ (or more generally $\Re z\leq E^+$),  the matrix operator
\[\mathscr{P}_{0,z}(\sigma)=\begin{pmatrix}
\overset{\circ}{n}_0(\sigma)-z&\Pi^*(\sigma)\\
\Pi(\sigma)&0
\end{pmatrix} : B^2(\R)\times\C\to L^2(\R)\times\C\,,\]
is bijective for all $\sigma\in\R$ and
\[\mathscr{P}_{0,z}(\sigma)^{-1}=\begin{pmatrix}
	\overset{\circ}{\mathfrak R}_{0,z}(\sigma)&\Pi^*\\
	\Pi&z-\overset{\circ}{\mu}^{[1]}_a(\sigma)
\end{pmatrix}=:\mathscr{Q}_{0,z}(\sigma)\,.\]
Moreover,  the operator symbols  $(s,\sigma) \mapsto \mathscr{P}_{0,z}(\sigma)$ and $(s,\sigma) \mapsto \mathscr{Q}_{0,z}(\sigma)$ belong to  $S\big(\R^2,\mathscr{L}( B^2(\R)\times\C, L^2(\R)\times\C)\big)$ and $
S\big(\R^2,\mathscr{L}( L^2(\R)\times\C, B^2(\R)\times\C)\big)$, respectively. We recall that $S(\R^2,F)$ is the set of smooth functions on $\R^2$, valued in $F$ with bounded derivatives (at any order).
\end{proposition}
\begin{proof}
By  straightforward computations, we can verify the identities  
\begin{equation*}
\mathscr Q_{0,z}(\sigma)\mathscr P_{0,z}(\sigma)={\rm Id}_{B^2(\R)\times \C} \mbox{  and } \mathscr P_{0,z}(\sigma)\mathscr Q_{0,z}(\sigma)={\rm Id}_{L^2(\R)\times\C}\,.
\end{equation*}
\end{proof}
\subsection{Some useful formulas}\label{sec:uf}
Let us   recall some formulas  and results  from \cite{AK20}.  Let $\phi_a$ be the \emph{positive} and $L^2$-normalized ground state of  the operator $\overset{\circ}{n}_0(\sigma(a))$. introduced  in \eqref{eq:circ-n}.   It is proven in \cite[Thm.~1.1]{AK20} that  $\phi_a'(0)<0$ for all $a\in(-1,a_0)$.

Some useful identities involve the moments
\begin{equation}\label{eq:moments}
M_n(a)=\int_{\mathbb R}\frac 1{b_a(\tau)}\big(b_a(\tau)\tau-\sigma(a)\big)^n|\phi_a(\tau)|^2\,d\tau\,, \end{equation}
 for $n\in \mathbb N $.
It has been proven in  \cite{AK20} that
	\begin{align}
	M_1(a)&=0\,,\label{eq:m1}\\
	M_2(a)&=-\frac 12 \beta_a\int_{\mathbb R}\frac 1{b_a(\tau)}|\phi_a(\tau)|^2\,d\tau+\frac 14\Big(\frac 1a-1\Big)\sigma(a)\phi_a(0)\phi_a'(0)\,, \label{eq:m2}\\
	M_3(a)&=\frac 13\Big(\frac 1a-1\Big)\sigma(a)\phi_a(0)\phi_a'(0)\,.\label{eq:m3*}
	\end{align}
The case $a=-1$ is special  because 
\[M_2(-1)=M_3(-1)=0\,,\]
while, for $-1<a<0$,   $M_3(a)<0$.

Finally, we will also need the following two identities \cite[Rem.~2.3]{AHK},
\begin{equation}\label{aqddid}
	\begin{split}
\int_{\mathbb R}\tau(\sigma(a)-b_a(\tau) \tau)^2|\phi_a(\tau)|^2\,\mathrm{d}\tau&=M_3(a) +\sigma(a)M_2(a)\,,\\
\int_{\mathbb R}b_a(\tau) \tau^2(\sigma(a)-b_a(\tau) \tau)|\phi_a(\tau)|^2\,\mathrm{d}\tau&=-M_3(a)-2\sigma(a)M_2(a)\,.
\end{split}
\end{equation}

\subsection{The symmetric case $a=-1$ and the de\,Gennes model}\label{sec.dG}
Let us recall the definition and properties of the  de\,Gennes model   occurring in the analysis of surface superconductivity within the Ginzburg-Landau model \cite[Sec.~3.2]{FH10} (and references therein).  We  start  with the family of harmonic oscillators 
\[\mathfrak  h[\sigma]=-\partial_t^2+(\sigma-bt)^2
\]
on the  half-axis $\R_+$ with Neumann condition at $0$.  Let us denote the positive normalized ground  state of $\mathfrak h[\sigma]$ by $f_{\sigma}$ and  the ground  state energy by $\mu(\sigma)$.   Then, minimizing with respect  to $\sigma\in\R$  we get
\[\Theta_0=\inf_{\sigma\in\R}\mu(\sigma)=\mu(\xi_0)~{\rm  where~}\xi_0=\sqrt{\Theta_0}\,.\]
Let $f_0:=f_{\xi_0}$.  Then,  for $a=-1$, we get by  a symmetry argument
\[ \phi_a(t)=f_0(|t|)\,
\qquad \text{ and } \sigma(-1) = \xi_0.\]

\subsubsection*{Moments}~\\
Let us introduce the following moments
\[M_k=\int_{\R_+}(\xi_0-t)^k|f_0(t)|^2dt\,.\]
Then,  by  \cite{FH06}, we have
\[M_0=1,\quad M_1=0,\quad  M_2=\frac{\Theta_0}2 ,\quad  M_3=-\frac{|f_0(0)|^2}{6}\,, \]
and
\[M_4= \frac{3}{8}\big(1+\Theta_0^2-\xi_0f_0(0)^2 \big)=\frac38(1+\Theta_0^2+6\xi_0M_3)\,.\]

\section{Decay  of bound states and spectral reduction}\label{sec:red}

 In this section, we consider the eigenfunctions of the operator $\mathcal P^a_h=\mathcal P_h$ with eigenvalues in the energy window 
\begin{equation}
J^+_h=[0,E^+h]\quad{\rm where~} E<E^+<|a|\,.
\end{equation}

We prove that the eigenfunctions associated with eigenvalues in $J^+_h$ are exponentially localized near $\Gamma$, see Corollary \ref{cor.expdecay}. To describe the effect of the edge on the localization, it is natural to use the classical tubular coordinates near $\Gamma$, whose definition will be recalled in Subsection \ref{sec.tubular}.  In order to prove Corollary \ref{cor.expdecay}, 
 we will have to  combine Agmon estimates and a rough estimate on the number of eigenvalues in $J^+_h$ (polynomially in $h^{-1})$), which will be discussed in  Subsection \ref{sec.number}.

\subsection{Tubular coordinates}\label{sec.tubular}
For all $\epsilon>0$, consider the $\epsilon$-neighborhood of $\Gamma$
\begin{equation}\label{eq:gam-ep}
	\Gamma(\epsilon)=\{x\in\mathbb R^2~:~\textrm{ dist}(x,\Gamma)<\epsilon\}\,.
\end{equation}
Consider a parameterization $\mathbf M(s)$ of the edge $\Gamma$  by the arc-length coordinate $s\in[-L,L)$,  where $L=|\Gamma|/2$.  Consider the unit normal $\mathbf n(s)$ to $\Gamma$ pointing inward to $\Omega_1$,  and the unit oriented tangent  $\mathbf t(s)=\dot{\mathbf n}(s)$ so that $(\mathbf t(s),\mathbf n(s))$ is a direct frame, i.e.  $\textrm{ det}(\mathbf t(s),\mathbf n(s))=1$.  We can now introduce the curvature  $k(s)$  at  the point $\mathbf M(s)$,  defined by $\ddot{\mathbf n}(s)=k(s)\mathbf n(s)$.

Let us represent the torus $(\mathbb R/ 2L \mathbb Z)$ by the interval $[-L,L)$ and pick $\epsilon_0>0$ so that 
\[\Phi: \mathbb R/ (2L \mathbb Z)\times (-\epsilon_0,\epsilon_0)\ni (s,t)\mapsto \mathbf{M}(s)+t\mathbf n(s)\in\Gamma(\epsilon_0) \]
is a diffeomorphism, with Jacobian
\begin{equation}\label{defm}m(s,t)=1-tk(s)\,.
\end{equation}
The Hilbert space $L^2(\Gamma(\epsilon_0))$ is transformed into the weighted space  
\[L^2\big((\mathbb R/ 2L \mathbb Z) \times(-\epsilon_0,\epsilon_0);m\,\mathrm{d}s\mathrm{d}t\big)\] and the  operator $\mathcal P_h$ is (locally near the edge) transformed into  the  following operator (see \cite[App.~F]{FH10})):
\begin{multline*}
	\widetilde{\mathcal P}_h:=-h^2m^{-1}\partial_t m\partial_t\\+m^{-1}\left(-ih\partial_s+\gamma_0-b_a(t)t+\frac{k}{2}b_a(t)t^2\right)m^{-1}\left(-ih\partial_s+\gamma_0-b_a(t)t+\frac{k}{2}b_a(t)t^2\right)
\end{multline*}
where $b_a$ is defined in \eqref{eq:ha} and
\begin{equation}\label{eq:circ}
	\gamma_0=\frac{|\Omega_1|}{2L}\,.
\end{equation}
\subsection{Number of eigenvalues}\label{sec.number}
We give a preliminary, rough bound on 
the number of eigenvalues in $J^+_h$. As we will see, this first estimate will be enough to deduce a stronger one at the end of our analysis.
\begin{proposition}\label{prop:Nb-h}
Let $N(\mathcal{P}_h, E^+h)={\rm Tr}\big(\mathbf 1_{J_h}(\mathcal P_h)\big)$. There exist $C,h_0>0$ such that, for all $h\in(0,h_0)$,  
\[N(\mathcal{P}_h, E^+h)\leq C h^{-2}\,.\]
\end{proposition}
\begin{proof}
Let us introduce  a fixed partition of the unity 
\[\chi_{\mathrm{out}}^2+\chi_{\mathrm{e}}^2+\chi_{\mathrm{in}}^2=1\,,\]	
such that $\mathrm{supp}(\chi_{\mathrm{out}})\subset\mathbb{R}^2\setminus\overline{\Omega}_1=\Omega_2$, $\mathrm{supp}(\chi_{\mathrm{in}})\subset\Omega_1$, and ${\mathrm{supp}(\chi_{\mathrm{e}})}\subset \overline{\Gamma(\epsilon_0)}$.
The quadratic form associated with $\mathcal{P}_h$ is given by	
\[Q_h(\psi)=\int_{\R^2}|(-ih\nabla+\mathbf{A})\psi|^2\mathrm{d}x\,,\]
for all $\psi\in L^2(\R^2)$ such that $(-ih\nabla+\mathbf{A})\psi\in L^2(\R^2)$.

The usual localization formula (see, for instance, \cite[Section 4.1.1]{Ray}) gives the existence of a constant $C>0$ such that
\[Q_h(\psi)\geq Q_h(\chi_{\mathrm{out}}\psi)+Q_h(\chi_{\mathrm{e}}\psi)+Q_h(\chi_{\mathrm{in}}\psi)-Ch^2\|\psi\|^2\,.\]
By noticing that $\psi\mapsto(\chi_{\mathrm{out}}\psi,\chi_{\mathrm{e}}\psi,\chi_{\mathrm{in}}\psi)$ is injective, and thanks to the min-max theorem, we find that
\[N(\mathcal{P}_h, E^+h)\leq N(\mathcal{P}^{\mathrm{out}}_h,E^+h+Ch^2)+N(\mathcal{P}_h^{\mathrm{e}},E^+h+Ch^2)+N(\mathcal{P}_h^{\mathrm{in}},E^+h+Ch^2)\,,\]
where the operators $\mathcal{P}^{\mathrm{out}}_h$, $\mathcal{P}^{\mathrm{e}}_h$ and $\mathcal{P}^{\mathrm{in}}_h$ are the Dirichlet realizations of $(-ih\nabla+\mathbf{A})^2$ on   $\Omega_2$,  $\Gamma(\epsilon_0)$ and $\Omega_1$,   respectively.
We recall that $E^+<|a|\leq |a_0|<1$ and notice that $\mathcal{P}^{\mathrm{out}}_h\geq |a|h$ and  $\mathcal{P}^{\mathrm{in}}_h\geq h$. When $h$ is small enough,  $E^+h+Ch^2<|a|h<h$,  so we must have   $N(\mathcal{P}^{\mathrm{out}}_h,E^+h+Ch^2)=N(\mathcal{P}^{\mathrm{in}}_h,E^+h+Ch^2)=0$. Thus,
\[N(\mathcal{P}_h, E^+h)\leq N(\mathcal{P}_h^{\mathrm{e}},E^+h+Ch^2)\,.\]
Therefore, we are reduced to estimate the number of eigenvalues of the operator with compact resolvent $\mathcal{P}_h^{\mathrm{e}}$ below $E^+h+Ch^2$. For that purpose, we can use the tubular coordinates and notice that, for all $\psi\in H^1_0(\Gamma(\epsilon_0))$,
\[Q_h(\psi)=\int_{(\mathbb R/ 2L \mathbb Z)\times(-\epsilon_0,\epsilon_0)}\left(|h\partial_t\psi|^2+m^{-2}|(hD_s+\gamma_0-b_a(t)t+\frac{k}{2}t^2)\psi|^2\right)m\mathrm{d}s\mathrm{d}t\,.\]
This gives the following rough estimate, for some $c_0,C_0>0$,
\[\frac{Q_h(\psi)}{\|\psi\|^2}\geq c_0\frac{\int_{(\mathbb R/ 2L \mathbb Z)\times(-\epsilon_0,\epsilon_0)} |h\partial_t\psi|^2+|h\partial_s\psi|^2\mathrm{d}s\mathrm{d}t}{\int_{(\mathbb R/ 2L \mathbb Z)\times(-\epsilon_0,\epsilon_0)} |\psi|^2\mathrm{d}s\mathrm{d}t}-C_0\,.\]
Thanks to the min-max theorem, this implies the upper bound
\[ N(\mathcal{P}_h^{\mathrm{e}},E^+h+Ch^2)\leq N(-h^2\Delta^{\mathrm{Dir}},c_0^{-1}(E^+h+Ch^2+C_0))\,,\]
where $-\Delta^{\mathrm{Dir}}$ is the Dirichlet Laplacian on the cylinder $(\mathbb R/ 2L \mathbb Z)\times(-\epsilon_0,\epsilon_0)$. The spectrum of this operator can be computed explicitely thanks to Fourier series, and we get the rough estimate
\[N(-h^2\Delta^{\mathrm{Dir}},c_0^{-1}(E^+h+Ch^2+C_0))\leq\tilde Ch^{-2}\,.\]
\end{proof}

Since $E^+<|a|$, the eigenfunctions of $\mathcal P_h$ associated with eigenvalues in the allowed energy window $J^+_h$ are localized near the edge, see \cite{AK20}.

\begin{proposition}\label{prop:dec}
There exist constants  $\alpha,h_0,C_0>0$ such that, if $h\in (0,h_0]$ and $u_h$ is an eigenfunction of $\mathcal P_h$ associated with an  eigenvalue in $J^+_h$, then the following holds,
\begin{equation}\label{eq:dec-norm}
\int_{\mathbb R^2}\big(|u_{h}|^2+h^{-1}|(-ih\nabla+\mathbf A)u_{h}|^2 \big)\exp\Big(\frac{2\alpha\, {\rm dist}(x,\Gamma)}{h^{1/2}} \Big)dx\leq C_0\|u_h\|^2_{L^2(\R^2)}\,.
\end{equation}
\end{proposition}

Combining Propositions \ref{prop:Nb-h} and \ref{prop:dec}, we get the following estimate.
\begin{corollary}\label{cor.expdecay}
Let $\eta\in\left(0,\frac12\right)$. There exists $h_0>0$ such that for all $h\in(0,h_0)$ and $u_h\in \mathrm{Ran}\mathds{1}_{J_h}(\mathcal{P}_h)$, we have outside $\Gamma(h^{\frac12-\eta})$
\begin{equation}\label{eq:dec-norm'}
	\int_{\mathbb{R}^2\setminus\Gamma(h^{\frac12-\eta})}\big(|u_{h}|^2+|(-ih\nabla+\mathbf A)u_{h}|^2 \big)\mathrm{d}x\leq e^{-h^{-\eta}}\|u_h\|^2_{L^2(\R^2)}\,.
\end{equation}
\end{corollary}
Corollary \ref{cor.expdecay} suggests to use the rescaling $t=\hbar \tilde t$. We also consider a smooth cutoff function 
\begin{equation}\label{eq:c-delta}
	c_\mu(\tilde t)= c(\mu\tilde t)\,,\quad \mu=h^{\eta}=\hbar^{2\eta}\,,
\end{equation}
where $c\in \mathscr{C}^\infty_0(\mathbb R)$ is even and satisfies $c=1$ on $[-1,1]$ and $c=0$ on $\mathbb R\setminus(-2,2)$. This cutoff function is convenient to define the new operator  on the Hilbert space\break$ L^2\big((\mathbb R/ 2L \mathbb Z)\times\mathbb R;m_\hbar  \mathrm{d}\tilde s \mathrm{d}\tilde t\big)$, by
\begin{equation*}
	\begin{split}
	\widetilde{\mathcal N}_{\hbar}&=-m^{-1}_{\hbar}\partial_{\tilde t} m_{\hbar}\partial_{\tilde t}\\
	&+m^{-1}_{\hbar}\left(\hbar D_{\tilde s}+\hbar^{-1}\gamma_0-b_a\tilde t +\hbar c_\mu \frac{k}{2}b_a\tilde t^2\right)  m^{-1}_{\hbar}\left(\hbar D_{\tilde s}+\hbar^{-1}\gamma_0-b_a\tilde t+\hbar c_\mu\frac{k}{2}b_a\tilde t^2\right)
	\end{split}
\end{equation*}
acting on the domain
\begin{align*}
\mathrm{Dom}(\widetilde{\mathcal N}_{\hbar})=\{ u\in L^2\big((\mathbb R/ 2L \mathbb Z)\times\mathbb R\big)& : \partial_{\tilde t}^2u\in L^2\big((\R/2L\Z)\times\mathbb R\big), \nonumber \\
& (\hbar D_{\tilde s}+\hbar^{-1}\gamma_0-b_a\tilde t )^2 u \in  L^2\big((\mathbb R/ 2L \mathbb Z)\times\mathbb R\big)   \}.
\end{align*}
As in Proposition \ref{prop:dec}, we can prove that the eigenfunctions of $\widetilde{\mathcal{N}}_h$ associated with eigenvalues in $J_h$ are localized near $\tilde t=0$.

\begin{proposition}
\label{prop:BHR2.7}
  The spectra of $\mathcal{P}_h$ and $\widetilde{\mathcal{N}}_h$ in $J^+_h$ coincide modulo $\mathcal{O}(h^{\infty})$.
\end{proposition}

Therefore, we are reduced to the spectral analysis of $\widetilde{\mathcal{N}}_h$. For shortness, we drop the tildes. Up to a change of gauge, we are reduced to the operator
\begin{equation*}\label{eq:tilde-Nh}
	\begin{split}
		{\mathcal N}_{\hbar,\theta}=&-m^{-1}_{\hbar}\partial_{t} m_{\hbar}\partial_{t}\\
		&+m^{-1}_{\hbar}\left(\hbar D_{ s}+\theta-b_a t +\hbar c_\mu \frac{k}{2}b_a t^2\right)  m^{-1}_{\hbar}\left(\hbar D_{ s}+\theta-b_a t+\hbar c_\mu\frac{k}{2}b_a t^2\right),
	\end{split}
\end{equation*}
with
$$
m_\hbar (\tilde s,\tilde t):= 1- \hbar c_\mu(\tilde s,\tilde t) \tilde t k(\tilde s)\,,
$$
and  domain
\begin{align*}
	\mathrm{Dom}({\mathcal N}_{\hbar,\theta})=\{ u\in L^2\big((\mathbb R/ 2L \mathbb Z)\times\mathbb R)& : \partial_{t}^2u\in L^2((\R/2L\Z)\times\mathbb R\big), \nonumber \\
	& (\hbar D_{s}+\theta-b_a t )^2 u \in  L^2(((\mathbb R/ 2L \mathbb Z))\times\mathbb R)   \}\,.
\end{align*}
Here
\begin{subequations}
\begin{equation}\label{eq:defthetaa}
\theta=\theta(\hbar)=\hbar^{-1}\gamma_0-\frac{m\pi}{L}\hbar \,
\end{equation}
where $m\in\Z$ is chosen so that 
\begin{equation}\label{eq:defthetb}
\theta(\hbar)\in[0,\hbar\pi L^{-1})\,.
\end{equation}
\end{subequations}

Before  going ahead,  we have to deal with the inconvenience of working in a Hilbert space with a weighted measure, which also depends on $\hbar$. Thus, let us use the canonical conjugation and work in the fixed Hilbert space with flat measure 
$L^2((\R/2L\Z)\times\R,\mathrm{d}s\mathrm{d}t)$:
\begin{equation}\label{eq:Nhc0}
	\begin{aligned}
		\mathfrak N_{\hbar,\theta}&=m_\hbar^{1/2}\mathcal N_{\hbar,\theta}m_\hbar^{-1/2}=-m^{-1/2}_{\hbar}\partial_t m_{\hbar}\partial_t m_{\hbar}^{-1/2}+\big(m_\hbar^{-1/2} \mathcal  T_\hbar m_\hbar^{-1/2}\big)^2
	\end{aligned}
\end{equation}
where
\begin{equation}\label{eq.Th}
\mathcal T_{\hbar,\theta}= \hbar D_s+\theta-b_at+\hbar c_\mu\frac{k}{2}b_at^2\,.
\end{equation}
Note that
\begin{equation}\label{eq.transverse}
-m^{-1/2}_{\hbar}\partial_t m_{\hbar}\partial_t m_{\hbar}^{-1/2}=-\partial^2_t-\frac{(\partial_tm_\hbar)^2}{4m^2_\hbar}+\frac{\partial^2_tm_\hbar}{2m_\hbar}\,,
\end{equation}
so that
\begin{equation}\label{eq:Nhc}
\mathfrak N_{\hbar,\theta}=-\partial^2_t-\frac{(\partial_tm_\hbar)^2}{4m^2_\hbar}+\frac{\partial^2_tm_\hbar}{2m_\hbar}+\big(m_\hbar^{-1/2} \mathcal  T_{\hbar,\theta} m_\hbar^{-1/2}\big)^2\,.
\end{equation}

We restate the Proposition~\ref{prop:BHR2.7} in terms of the new notation.
\begin{proposition}\label{prop:BHR2.7-bis}
  The spectra of $\mathcal{P}_h$ and $\mathfrak N_{\hbar,\theta}$ in $J^+_h$ coincide modulo $\mathcal{O}(h^{\infty})$.
\end{proposition}

\section{A pseudodifferential operator with operator valued symbol}\label{sec:pseudo}

\subsection{Preliminaries}\label{s-sec.p}
 Let us briefly prove  that  an operator  given by \eqref{eq.ph-w-intro} with a $2L$ periodic symbol,  $p_{\hbar}(s+2L,\sigma)=p_{\hbar}(s,\sigma)$,    preserves $2L$-periodic distributions and  also locally square  integrable $2L$-periodic functions. More generally, it also preserves the set of functions
 \begin{equation}\label{eq.Fthetah}
 \mathscr{F}_{\hbar,\theta}:=\{u\in L^2_{loc}(\mathbb{R}) : u(s+2L)=e^{2i\theta L/\hbar}u(s)\}\,,
\end{equation}
equipped with the $L^2$-norm on a period $[-L,L)$. The operator ${\rm Op}_\hbar^w(p_\hbar)$ acts continuously on $\mathscr{F}_{\hbar,\theta}$. 
In fact, this is even true in the vector valued case where we replace $u\in L^2_{loc}(\mathbb{R})$ by $u\in L^2_{loc}(\mathbb{R};F)$ for some Hilbert space $F$ in the definition of $ \mathscr{F}_{\hbar,\theta}$.

Let us explain this for $\theta=0$.

  From the composition theorem for pseudodifferential operators (see \cite[Theorem 4.18]{Zworski}), we see that $\langle x\rangle^{-1}{\rm Op}_\hbar^w(p_\hbar)\langle x\rangle$ is a pseudodifferential operator with symbol in $S(1)$ (and thus it is bounded on $L^2(\R)$ thanks to the Calder\'on-Vaillancourt theorem, see \cite[Theorem 4.23]{Zworski}). This shows that ${\rm Op}_\hbar^w(p_\hbar)$ is bounded from $L^2(\R,\langle x\rangle^{-2}\mathrm{d}x)$ to $L^2(\R,\langle x\rangle^{-2}\mathrm{d}x)$. Notice that there exist $C_1(L)>0$, $C_2(L)>0$ and $C_3(L)>0$ such that for all $u\in L^2_{2L}(\R)$,
\begin{multline*}
\|u\|^2_{L^2(\R,\langle x\rangle^{-2}\mathrm{d}x)} \leq C_1(L)  \sum_{\ell \in \mathbb Z} \langle\ell\rangle^{-2} \| u\|^2_{L^2(2 \ell L  - L ,2\ell L + L)} \leq C_2(L) \|u\|^2_{L^2_{2L}(\R)}\\
\leq C_3(L)\|u\|^2_{L^2(\R,\langle x\rangle^{-2}\mathrm{d}x)} \,.
\end{multline*} 
Now  the operator $\mathfrak{N}_{\hbar,\theta}$ introduced in \eqref{eq:Nhc} (with $\mathcal T_{\hbar,\theta}$ introduced in \eqref{eq.Th}) can be seen as the action of an $\hbar$-pseudodifferential operator $\mathfrak{N}_\hbar$ with operator symbol $n_{\hbar}$ on $\mathscr{F}_{\hbar,\theta(\hbar)}$ where $\theta(\hbar)$ is defined in \eqref{eq:defthetb}. We have
\begin{equation}\label{eq:N0}
	\mathfrak N_{\hbar}=-\partial^2_t-\frac{(\partial_tm_\hbar)^2}{4m^2_\hbar}+\frac{\partial^2_tm_\hbar}{2m_\hbar}+\big(m_\hbar^{-1/2} \mathcal  T_{\hbar} m_\hbar^{-1/2}\big)^2\,,
\end{equation}
with
\[\mathcal T_{\hbar}= \hbar D_s-b_at+\hbar c_\mu\frac{k}{2}b_at^2\,.\]
We recall the classical notation for the Weyl quantization
\begin{equation}\label{eq:Nhc-symb}
\mathfrak{N}_\hbar=\mathrm{Op}^w_\hbar (n_{\hbar}) u (s)=\frac{1}{(2\pi\hbar)} \int_{\mathbb  R^2} e^{i(s-\tilde s)\cdot \sigma/\hbar} n_{\hbar} \left(\frac{s+\tilde s}{2}, \sigma\right)u(\tilde s)   \mathrm{d}\tilde s \mathrm{d}\sigma\,.
\end{equation}
Note that, by using the Floquet-Bloch transform, $\mathfrak{N}_\hbar$ is unitarily equivalent to the direct integral of the $\mathfrak{N}_{\hbar,\theta}$.

Let us explain why the operator $\mathfrak{N}_\hbar$ can be written under the form \eqref{eq:Nhc-symb}. Note already that, at a formal level, we expect that 
\[n_{\hbar}\simeq n_0=-\partial^2_t+(\sigma-b_a t)^2\,.\]
This formal principal symbol suggests to consider the set of operator-valued symbols (where lies $n_0$).  We denote by $S\big(\R^2,\mathscr{L}(B^2(\R)\times\C,L^2(\R)\times\C)\big)\,$ the class of symbols $\Psi$ on $\mathbb R^2$
with value in $\mathscr{L}(B^2(\R)\times\C,L^2(\R)\times\C)$,   such that, for all $j,k \in \mathbb N$, there exists $C_{j,k} >0$
 $$
 \| \partial_s^j\partial_\sigma^k \Psi (s,\sigma)\|_{\mathscr{L} (B^2(\mathbb R)\times \mathbb C),L^2(\mathbb R)\times \mathbb C))}\leq C_{j,k}\,,
 $$
where in the definition on the norms above the norm on $B^2(\R)$ is $(s,\sigma)$-dependent and  given by
\[
B^2(\mathbb R)\ni \psi \mapsto \|\psi\|^2_{B^2(\R),(s,\sigma)}=\|\partial^2_t\psi\|^2+\|\langle t^2\rangle\psi\|^2+\langle \sigma\rangle^4\|\psi\|^2\,.\]
Here we used the notation
\begin{equation}\label{eq:<u>}
\langle u\rangle=(1+|u|^2)^{1/2}\,.
\end{equation} 
Our symbols can be $\hbar$-dependent and in this case we impose above the uniformity of the constants with respect to $\hbar$.\/
The representation of  $\mathfrak{N}_{\hbar}$ as a pseudo-differential operator follows from the results of composition for operator symbols (see \cite[Theorem 2.1.12]{Keraval}) and by noticing that the symbol of \eqref{eq.Th} (obtained by replacing $\hbar D_s$ by $\sigma$) belongs to $S\big(\R^2,\mathscr{L}(B^1(\R)\times\C,L^2(\R)\times\C)\big)$ and also to $S\big(\R^2,\mathscr{L}(B^2(\R)\times\C,B^1(\R)\times\C)\big)$ (with  suitable $(s,\sigma)$-dependent norms on $B^k(\R)$
 extending the definition given above  of the norm on $B^2$).  Indeed, the function $t\mapsto\hbar b_a c_\mu t^2$ is bounded, uniformly in $\hbar$, since $\mu=\hbar^{2\eta}$ (for $\eta$ fixed small enough).

\begin{remark}
We recall that the operator and its Weyl symbol are related by the following exact formula (see  for instance   \cite[Theorems 4.19 \& 4.13]{Zworski} whose proof can be adapted to operator-valued symbols):
\[n_{\hbar}(s,\sigma)=e^{-i\frac\hbar2 D_s D_\sigma}\left[e^{-is\sigma/\hbar}\mathfrak{N}_{\hbar} (e^{i\cdot\sigma/\hbar})\right]\left(s,\sigma\right)\,,\]
where $e^{-i\frac\hbar2 D_s D_\sigma}$ is defined as a Fourier multiplier thanks to the Fourier transform with respect to $(s,\sigma)$.
\end{remark}

\subsection{Expansion of $\mathfrak{N}_{\hbar}$}
Let us now describe an expansion of $n_{\hbar}$ in powers of $\hbar$. We would like to write 
\begin{equation}\label{eq:exp-symb-Nhc}
	n_{\hbar}\simeq n_0+\hbar n_1+\hbar^2n_2+\ldots
\end{equation}
With this writing,  we mean an expansion of the associated operator $\mathfrak  N_{\hbar}$ of the following form
\begin{equation}\label{eq:ex-Nhc}
	\mathfrak{N}_{\hbar}=\mathfrak n_0+\hbar\mathfrak n_1+\hbar^2\mathfrak n_2+\mathcal \hbar^3\mathscr R_\hbar^{(3)}+\hbar w_\hbar\,,
\end{equation}
where, for some $N\in\mathbb {N}$, $C, \hbar_0>0$, we have, for all $\hbar\in(0,\hbar_0)$,

\begin{enumerate}[\rm (i)]
\item\label{eq.i} $w_{\hbar}$ is a smooth function supported in $\{(s,t) : C^{-1}\hbar^{-2\eta}\leq \langle t\rangle\leq C\hbar^{-2\eta}\}$ and such that $w_\hbar=\mathscr{O}(\langle t\rangle)$ ,
\item \label{eq.ii}$\mathscr R_\hbar^{(3)}$  is a pseudodifferential operator whose symbol belongs to a bounded set in $S(\R^2,\mathscr{L}(B^2(\R)\times\C,L^2(\R,\langle t\rangle^{-N}\mathrm{d}t)\times\C))$.
\end{enumerate}
Note that \eqref{eq:exp-symb-Nhc} does not mean  an expansion in the symbol class $S(\R^2,\mathscr{L}(B^2(\R)\times\C,L^2(\R)\times\C))$, where lies $n_{\hbar}$.   
We start by expanding the differential operator $\mathfrak{N}_{\hbar}$ (see \eqref{eq:Nhc}) with respect to $\hbar$, with $\mu$ (involved in the cutoff functions $c_\mu$) considered as a parameter\footnote{Note that $\hbar c_{\mu}(t)t$ converges to $0$ uniformly  as $\hbar$ tends to $0$ since $\mu=\hbar^{2\eta}$ and $\eta<1/2$.}.

In the following proposition, we describe the (symmetric) differential operators $\mathfrak{n}_j$.
\begin{proposition}
The decomposition \eqref{eq:ex-Nhc} holds with
	\begin{equation}\label{eq:ex-terms}
	\begin{aligned}
		\mathfrak n_0&=-\partial_t^2+\mathfrak{p}_0^2\,,\\
		\mathfrak n_1&=\mathfrak p_0\mathfrak p_1+\mathfrak p_1\mathfrak p_0\,,\\
		\mathfrak n_2&=\mathfrak p_0\mathfrak p_2+\mathfrak p_1^2+\mathfrak p_2\mathfrak p_0-c_\mu^2\frac{k^2}{4}\,,
	\end{aligned}
\end{equation}	
where 
	\begin{equation}\label{eq:D-theta}
		\begin{aligned}
	\mathfrak{p}_0&=\hbar D_s-b_a t\,,\\
		\mathfrak{p}_1 &=c_\mu t\Big(\frac{k}2\mathfrak p_0+\frac12\mathfrak  p_0k+\frac{k}2 b_at \Big)\,,\\
		\mathfrak{p}_2&=c_\mu^2t^2\Big(\frac{k^2}{2}\mathfrak p_0+\frac12\mathfrak p_0k^2+\frac{k^2}2b_at \Big)\,.
	\end{aligned}
\end{equation}
\end{proposition}

\begin{proof}
Let us provide a Taylor expansion of \eqref{eq:Nhc}.  \eqref{eq.Th} can be rewritten in the form
	\[\mathcal T_{\hbar}=	\mathfrak{p}_0+\hbar c_\mu\frac{k}{2}b_at^2\,.\]
Straightforward  computations yield,
	\[\begin{aligned}
		m_\hbar^{-1/2}\mathcal T_{\hbar} m_\hbar^{-1/2}&=m_\hbar^{-1}\mathfrak p_0+m_\hbar^{-1/2}\big(\hbar D_s m_\hbar^{-1/2}\big)+\hbar c_\mu\frac{k}{2}b_at^2 m_\hbar^{-1}\\
		&=m_\hbar^{-1}\mathfrak p_0+\hbar^2 m_\hbar^{-2} c_\mu t(D_sk)+\hbar m_\hbar^{-1}c_\mu\frac{k}{2}b_at^2 \,.
	\end{aligned}\]
	Now  we expand $m_\hbar^{-1}$  in powers of $\hbar$  and get
	\[  m_\hbar^{-1}=1+\hbar  c_\mu tk+\hbar^2c_\mu^2t^2k^2+\hbar^3\frac{(c_\mu  t k)^3}{1-\hbar c_\mu t k}\,,\]
	so that
	\[\begin{aligned}
		\hbar^2m_\hbar^{-2}&=\hbar^2+\hbar^3\frac{2c_\mu tk}{1-\hbar c_\mu tk}+\hbar^4\frac{c_\mu^2t^2k^2}{(1-\hbar c_\mu tk )^2}\,,\\
		\quad
		\hbar m_\hbar^{-1}&=\hbar+\hbar^2  c_\mu tk+\hbar^3\frac{c_\mu^2t^2k^2}{1-\hbar c_\mu t k}\,. 
	\end{aligned}
	\]

	We  have the following expansion 
	\begin{multline*}
		m_\hbar^{-1/2} \mathcal  T_{\hbar} m_\hbar^{-1/2}=\\ =	\mathfrak{p}_0+\hbar 	kc_\mu t\Big(\mathfrak p_0+\frac12 b_at \Big)
		+\hbar^2	\left( k^2c_\mu^2t^2\Big(\mathfrak p_0+\frac12b_at \Big)+
		\frac12c_\mu t(D_sk)\right)\\
		+\frac{\mathcal \hbar^3c_\mu  t k}{1-\hbar c_\mu t k}\left( (c_\mu  t k)^2 \mathfrak p_0+\Big( 2  +\frac{\hbar c_\mu t k }{1-\hbar c_\mu tk }\Big)c_\mu t(D_sk)+
		\frac{2c_\mu^2 t^3k^2}{2}\right)\,.
	\end{multline*}
	The  previous  expression can be rearranged as follows
	\[
	m_\hbar^{-1/2} \mathcal  T_{\hbar} m_\hbar^{-1/2}=	\mathfrak{p}_0+\hbar 	\mathfrak{p}_1 +\hbar^2	\mathfrak{p}_2+\mathcal \hbar^3\mathscr R_\hbar\,.
	\]
	and
	\begin{equation}\label{eq.remainder}
		\begin{aligned}
			\mathscr R_\hbar&=
			-\hbar c_\mu^2t^2(D_sk)k\\
			&\qquad+\frac{c_\mu  t k}{1-\hbar c_\mu t k}\left( (c_\mu  t k)^2 \mathfrak p_0+\Big( 2  +\frac{\hbar c_\mu t k }{1-\hbar c_\mu tk }\Big)c_\mu t(D_sk)+
			\frac{2c_\mu^2 t^3k^2}{2}\right)\,.
		\end{aligned}
	\end{equation}
Recalling \eqref{eq.transverse}, we can also  expand the operator in the transversal variable and get
\begin{subequations}
	\begin{equation}\label{vhwha}
		-m^{-1/2}_{\hbar}\partial_t m_{\hbar}\partial_t m_{\hbar}^{-1/2}=-\partial_t^2-\hbar^2c^2_\mu\frac{k^2}{4} +\hbar^3v_\hbar
		+\hbar w_\hbar\,,
	\end{equation}
where the functions $v_\hbar$ and $w_\hbar$ satisfy, uniformly with respect to $s$,   and $\hbar$,
\begin{equation}\label{vhwhb}
v_\hbar(s,t)=\mathscr O\big( \langle t\rangle^4\big) \mbox{  and }   w_\hbar(s,t)=\mathscr O\big(| c_\mu''t|+|\langle t\rangle c_\mu'|\big)\,,\end{equation}
which gives in particular \eqref{eq.i}.
\end{subequations}

We get the expansion of the operator in \eqref{eq:ex-Nhc}
	and the remainder term is expressed via $\mathscr R_\hbar$ in \eqref{eq.remainder} as follows
	\begin{multline*}
		\mathscr R_\hbar^{(3)}=\mathfrak p_1\mathfrak p_2+\mathfrak p_2\mathfrak p_1 +\mathfrak p_0\mathscr R_\hbar +\mathscr R_\hbar\mathfrak p_0\\
		+\hbar\big( \mathfrak p_1\mathscr R_\hbar +\mathscr R_\hbar\mathfrak p_1+\mathfrak p_2^2\big)+\hbar^2\big(\mathfrak p_2\mathscr R_\hbar +\mathscr R_\hbar\mathfrak p_2 \big)
		+\mathscr O\big(\hbar^3\langle t\rangle^4\big)\,.
	\end{multline*}
 We see that the remainder $\mathscr R_\hbar^{(3)}$ satisfies \eqref{eq.ii}.
\end{proof}
We can now establish an expansion of the form \eqref{eq:exp-symb-Nhc} by considering the Weyl symbols of the $\mathfrak{p}_j$ in \eqref{eq:D-theta} (and the composition of pseudodifferential operators). We get the decomposition
\begin{subequations}\label{eq.nj}
\begin{equation}
n_{\hbar}=n_0+\hbar n_1+\hbar^2 n_2+\hbar^3 r_{3,\hbar}+w_\hbar\,,
\end{equation}
where
\begin{equation}\label{eq.n0,1,2}
\begin{split}
	n_0(s,\sigma)&=-\partial^2_t+(\sigma-b_a t)^2\,,\\
	n_1(s,\sigma)&=c_\mu k(s)\big(2t(\sigma-b_a t)^2+b_a t^2(\sigma-b_a t)\big)\,,\\
	n_2(s,\sigma)&= c^2_\mu k(s)^2\Big( 3t^2 (\sigma-b_a t)^2 +2 b_at^3 (\sigma-b_a t)+\frac14 b_a^2t^4 \Big) 
	-c_\mu^2\frac{k(s)^2}4\,,
\end{split}
\end{equation}
\begin{equation} r_{\hbar,3}\in S(\R^2,\mathscr{L}(B^2(\R)\times\C,L^2(\R,\langle t\rangle^{-N}\mathrm{d}t)\times\C))\end{equation}   and $w_\hbar$ is introduced in (\ref{vhwhb}). 
\end{subequations}

\section{The Grushin reduction}\label{sec:Gr}
 Instead of the operator $\mathfrak{N}_\hbar$, we consider its truncated version defined by 
\begin{equation}\label{eq:op-c-N}
	\mathfrak{N}^c_{\hbar}= \mathrm{Op}^w_\hbar (n^c_{\hbar})\,,\quad n^c_{\hbar,0}(s,\sigma)= n_{\hbar }(s,\chi_1(\sigma))\,,
\end{equation}
where  $\chi_1$ is defined in Section \ref{sec.chi1}.

Consider the operator symbol, for all $z\in[0,E]$ and $E<E^+<|a|$. 
\begin{equation}\label{eq:Gr1}
\mathscr{P}_{\hbar,z}(s,\sigma)=\begin{pmatrix}
n^c_{\hbar}-z&\Pi^*_\sigma\\
\Pi_\sigma&0
\end{pmatrix}=\mathscr{P}_{0,z}+\hbar\mathscr{P}_1+\hbar^2\mathscr{P}_2+\ldots\,,\end{equation}
where,  $\Pi_{\sigma}=\langle \cdot,\overset{\circ}{u}_\sigma \rangle$ and for all $j\geq 1$,
\begin{equation}\label{eq:Gr2}
\mathscr{P}_j=\begin{pmatrix}
 n^c_j&0\\
0&0
\end{pmatrix}\,,\quad  n_j^c(s,\sigma)=n_j\big(s,\chi_1(\sigma)\big)\,.
\end{equation}
The operator $\mathscr  P_{0,z}$ is   introduced  in  Proposition~\ref{prop.Grushin0}.  Recall that  it is bijective (since $z\in[0,E]$)  and  
\begin{equation}\label{eq:Gr3}
\mathscr  P_{0,z}^{-1}=\mathscr Q_{0,z}= \left(\begin{array}{cc}
q_{0,z}&q_{0,z}^+\\
q_{0,z}^-&q_{0,z}^\pm
\end{array}\right)
\end{equation}
is explicitly given in  Proposition~\ref{prop.Grushin0}.

\begin{proposition}\label{prop.param}
Consider
\[\mathscr{Q}_{1,z}=-\mathscr{Q}_{0,z}\mathscr{P}_1\mathscr{Q}_{0,z}=  \left(\begin{array}{cc}
q_{1,z}&q_{1,z}^+\\
q_{1,z}^-&q_{1,z}^\pm
\end{array}\right)\]
and
\[\mathscr{Q}_{2,z}= -\mathscr Q_1\mathscr P_1\mathscr Q_0-\mathscr Q_0\mathscr P_2\mathscr Q_0=\left(\begin{array}{cc}
q_{2,z}&q_{2,z}^+\\
q_{2,z}^-&q_{2,z}^\pm
\end{array}\right)\,.\]
We let
\[\mathscr{Q}_{\hbar}(z)=\mathscr{Q}_{0,z}+\hbar\mathscr{Q}_{1,z}+\hbar^2\mathscr{Q}_{2,z}
=\left(\begin{array}{cc}
q_{\hbar,z}&q^+_{\hbar,z}\\
q^-_{\hbar,z}&q^\pm_{\hbar,z}
\end{array}\right)\,. \]
Then,  
\begin{align*}
\mathrm{Op}^w_\hbar (\mathscr{Q}_\hbar(z))\,  \mathrm{Op}^w_\hbar (\mathscr{P}_\hbar(z)) &=\mathrm{Id}+\hbar^3\mathscr E_{\hbar,l}\,,\\
\mathrm{Op}^w_\hbar(\mathscr{P}_\hbar(z))\, \mathrm{Op}^w_\hbar( \mathscr{Q}_\hbar(z)) &=\mathrm{Id}+ \hbar^3\mathscr E_{\hbar,r}\,,\end{align*}
where  $\mathscr E_{\hbar,l/r}$ is a pseudodifferential operator, whose operator-valued symbol belongs to the class $S\big(\R^2,\mathscr{L}\big(L^2(\R)\times\C, L^2(\R,\langle t\rangle^{-N}\dd t)\times\C\big)\big)$, uniformly in $\hbar$, for some $N\in\mathbb{N}$ independent of $\hbar$,

\end{proposition}
The coefficients appearing in Proposition~\ref{prop.param} can be computed explicitly.  Of particular importance to us is 
\begin{equation}\label{eq:param1}
q^\pm_{\hbar,z}(s,\sigma)=z- \overset{\circ}{\mu}_a(\sigma)+\hbar  q_{1}^\pm(s,\sigma)+\hbar^2 q^\pm_{2,z}(s,\sigma)\,,
\end{equation}
where
\begin{align}
  q_1^\pm(s,\sigma)&=-\langle n_1(s,\chi_1(\sigma))\overset{\circ}{u}_{\sigma},\overset{\circ}{u}_{\sigma}\rangle\,, \label{eq:param20}\\
  q_{2,z}^\pm(s,\sigma)&=\langle q_{0,z}n_1(s,\chi_1(\sigma))\overset{\circ}{u}_{\sigma},n_1(s,\chi_1(\sigma))\overset{\circ}{u}_{\sigma}\rangle-\langle n_2(s,\chi_1(\sigma))\overset{\circ}{u}_{\sigma},\overset{\circ}{u}_{\sigma}\rangle\,.\label{eq:param2}
 \end{align}
Here $\overset{\circ}{u}_\sigma$ is  the positive ground state of the operator  in \eqref{eq:circ-n} and $n_0,n_1,n_2$ are introduced in \eqref{eq.n0,1,2}.

\begin{proposition}\label{cor.1}
Writing 
\[
\mathrm{Op}^w_\hbar\big(\mathscr{Q}_\hbar(z)\big)=\begin{pmatrix}
	Q_\hbar&Q^+_\hbar\\
	Q^-_\hbar&Q^\pm_\hbar
\end{pmatrix}\,,\qquad \mathfrak{P}_\hbar= \mathrm{Op}^w_\hbar(\Pi)\,,\]
we have
\begin{equation}\label{eq:cor.1-1}
Q_{\hbar}(\mathfrak{N}_{\hbar}^c-z)+Q^+_\hbar\mathfrak{P}_\hbar=\mathrm{Id}+{\hbar^3\mathscr R_{\hbar}^+}\,,\quad Q^-_\hbar(\mathfrak{N}_{\hbar}^c-z)+Q^\pm_\hbar\mathfrak{P}_\hbar=\hbar^3\mathscr R_{\hbar}^\pm\,,
\end{equation}
\begin{equation}\label{eq:cor.1-2}Q^-_\hbar=\mathfrak{P}_\hbar+{ \hbar\mathscr E_\hbar^-}\,,\quad Q^+_\hbar=\mathfrak{P}^*_\hbar+\hbar\mathscr E_\hbar^+\,,
\end{equation}
where $ \mathscr R_{\hbar}^+,\mathscr R_{\hbar}^\pm$ are pseudodifferential operators whose symbols belong to the class  $S\big(\R^2,\mathscr{L}\big(L^2(\R)\times\C, L^2(\R, \langle t\rangle^{-N}\mathrm{d}t)\times\C\big)\big)$, and where $\mathscr E_{\hbar}^-,\mathscr E_{\hbar}^+$ are pseudodifferential operators whose symbols belong to the class  $S\big(\R^2,\mathscr{L}\big(L^2(\R)\times\C, L^2(\R)\times\C\big)\big)$, uniformly in $\hbar$.
\end{proposition}

\section{Spectral applications}

\subsection{Localization of the eigenfunctions of $\mathfrak{N}_{\hbar,\theta}$}
In order to perform the spectral analysis of $\mathfrak{N}_{\hbar,\theta}$, we need to prove that its eigenfunctions (associated with eigenvalues in $[0,E^+]$) are $\hbar$-microlocalized, with respect to $\sigma+\theta$ in
\[\{\varsigma\in\R : \mu_a(\varsigma)\leq E^++\epsilon\}\,,\quad \mbox{ with }  \epsilon>0\, \mbox{ such that }E^++\epsilon<a\,.\]
This can be formulated in terms of the semiclassical wavefront/frequency set (see \cite[Sec.~8.4.2, ~p.188]{Zworski}),  however we write a stronger estimate in Proposition \ref{prop:mlc} below which holds uniformly with respect to $\theta\in\mathbb{R}$.  This is a consequence of  the behavior of the principal operator symbol $n_{0,\theta}=-\partial^2_t+(\sigma+\theta+b_a t)^2$ (which appears after the Bloch-Floquet transform), which is bounded from below by $\mu_a(\sigma+\theta)$.  

The following  estimate holds (see \cite[Section 5]{BHR22} where similar considerations are described in detail).
\begin{proposition}\label{prop:mlc}
	Consider a smooth function $\chi$ that equals $1$ away from $\{ \mu_a\leq E^++\epsilon\}$ and $0$ on $\{\mu_a\leq E^++\frac{\epsilon}{2}\}$. Then, for any $\theta\in\mathbb{R}$ and any normalized eigenfunction $\psi$ of  the  operator $\mathfrak N_{\hbar, \theta}$  associated with an eigenvalue in $[0,E^+]$,  we have
	\begin{equation}\label{eq.chiM}
		{\rm Op}^w_\hbar ( \chi(\cdot+\theta))\psi=\mathscr{O}(\hbar^\infty)\,,
	\end{equation}
	uniformly with respect to $\theta\in\mathbb{R}$,  where $\mathscr{O}(\hbar^\infty)$ holds in the sense of the norm $u\mapsto \|\langle t\rangle^2 u\|_{H^2((\R/2L\Z)\times\R_+)}$.  In addition,  \eqref{eq.chiM} also holds for all normalized $\psi\in \mathrm{Ran}\mathds{1}_{[0,E^+]}(\mathfrak{N}_{\hbar,\theta})$. 
\end{proposition}

Let us consider the operator $\mathfrak{N}^c_{\hbar,\theta}$ (with periodic boundary conditions) defined as the operator induced by   $\mathfrak{N}^c_{\hbar}$ on $\mathscr{F}_{\hbar, \theta} $ (defined in \eqref{eq.Fthetah}). By using  Proposition \ref{prop:mlc} and the min-max theorem, we get the following.
\begin{proposition}\label{prop.coincide}
	The spectra of $\mathfrak{N}_{\hbar,\theta}$ and $\mathfrak{N}^c_{\hbar,\theta}$ in $[0,E_+]$ coincide (with multiplicity) modulo $\mathscr{O}(\hbar^\infty)$,  uniformly with respect to $\theta\in\mathbb{R}$.   More precisely,  for all $N\geq 1$, there exist $\hbar_0, C>0$ such that, for all $\theta\in\mathbb{R}$ and all $\hbar\in(0,\hbar_0)$ and all $k\geq 1$ such that $\lambda_k(\mathfrak{N}_{\hbar,\theta})\leq E_+$, we have
	\[|\lambda_k(\mathfrak{N}_{\hbar,\theta})-\lambda_k(\mathfrak{N}^c_{\hbar,\theta})|\leq Ch^N\,.\]
\end{proposition}

\subsection{Weyl estimate}
A remarkable consequence of Proposition \ref{prop.param} and its corollary is the following Weyl estimate, which improves Proposition \ref{prop:Nb-h}. 
\begin{proposition}\label{cor.Weyl} 
Let $\theta=\theta(\hbar)$ be  as  defined in \eqref{eq:defthetaa}. For $E\in(0,|a|)$,  we have as $\hbar\to0$,
\[N(\mathfrak{N}_{\hbar,\theta},E)\underset{\hbar\to 0}\sim N(\mathfrak{N}^c_{\hbar,\theta},E)\underset{\hbar\to 0}\sim \frac{L(\sigma_+(a,E)-\sigma_-(a,E))}{\pi\hbar}\,,\]
where $\sigma_\pm(a,E)$ is defined in \eqref{eq:xi-pm}.
In particular,
\[N(\mathcal{P}_h, Eh)\underset{\hbar\to 0}\sim \frac{L(\sigma_+(a,E)-\sigma_-(a,E))}{\pi\sqrt{h}}\,.\]
\end{proposition}
\begin{proof}
The first asymptotics,  $N(\mathfrak{N}_{\hbar,\theta},E)\underset{\hbar\to 0}\sim N(\mathfrak{N}^c_{\hbar,\theta},E)$,  follows from Proposition~\ref{prop.coincide}. 
	Let us focus on establishing  the second one.

Note that  $Q^\pm={\rm Op}(q_{h}^\pm(z))$ with $q_h^{\pm}(z)$ given in \eqref{eq:param1}. Let us now test \eqref{eq:cor.1-1} (with $z=0$) and \eqref{eq:cor.1-2} with functions of $\mathscr{F}_{\hbar, \theta}$ of the form $u=e^{i\theta s/\hbar}\psi$, with  $\psi$ in  the domain of the  operator $\mathfrak{N}_{\hbar,\theta}^c$ (with periodic conditions). We get
\begin{equation}\label{eq.Qpm}
Q^-_{\hbar,\theta}\mathfrak{N}_{\hbar,\theta}^c\psi+Q^\pm_{\hbar,\theta}\mathfrak{P}_{\hbar,\theta}\psi=\hbar^3\mathscr R_{\hbar,\theta}^\pm\psi\,.
\end{equation}
where the index $\theta$ refers to the conjugation by $e^{i\theta s/\hbar}$ (or the translation by $\theta$ of the symbol in $\sigma$). Then, we take the inner product with $\mathfrak{P}_{\hbar,\theta}\psi$.  To deal with the term involving $Q^-_{\hbar,\theta}$, we use the first equality in \eqref{eq:cor.1-2}, and this gives
\begin{multline}\label{eq.ubqpm}
\langle{\rm Op}^w_\hbar\big(\overset{\circ}{\mu}_a(\cdot+\theta)\big)\mathfrak{P}_{\hbar,\theta}\psi,\mathfrak{P}_{\hbar,\theta}\psi\rangle\leq\Re \langle \mathfrak{N}^c_{\hbar,\theta}\psi,\mathfrak{P}^*_{\hbar,\theta}\mathfrak{P}_{\hbar,\theta}\psi\rangle
+C\hbar\|\mathfrak{N}^c_{\hbar,\theta}\psi\|\|\psi\|\\
+( C\hbar^3\|\langle t\rangle ^N\psi\|+C\hbar\|\psi\|)\|\mathfrak{P}_{\hbar,\theta}\psi\|\,.
\end{multline}
We apply this inequality to $\psi$ being a linear combination of eigenfunctions of $\mathfrak N_{\hbar,\theta}^c$ associated with eigenvalues less than $E$ and thus, thanks to the Agmon estimates (with respect to $t$), we can write, for some $C_0>0$, $\eta\in(0,1)$ and for $\hbar$ small enough,
\[\langle{\rm Op}^w_\hbar\big(\overset{\circ}{\mu}_a(\cdot+\theta)\big)\mathfrak{P}_{\hbar,\theta}\psi,\mathfrak{P}_{\hbar,\theta}\psi\rangle\leq \|\mathfrak{N}^c_{\hbar,\theta}\psi\|(\|\mathfrak{P}^*_{\hbar,\theta}\mathfrak{P}_{\hbar,\theta}\psi\|+C\hbar\|\psi\|) +C\hbar\|\psi\|\|\mathfrak{P}_{\hbar,\theta}\psi\|\,.\]
By definition of $\mathfrak{P}_{\hbar,\theta}$, we see that the principal symbol of $\mathfrak{P}^*_{\hbar,\theta}\mathfrak{P}_{\hbar,\theta}$ is a projection so that 
\[\langle{\rm Op}^w_\hbar\big(\overset{\circ}{\mu}_a(\cdot+\theta)\big)\mathfrak{P}_{\hbar,\theta}\psi,\mathfrak{P}_{\hbar,\theta}\psi\rangle\leq(1+\tilde C\hbar) \|\mathfrak{N}^c_{\hbar,\theta}\psi\| \|\psi\|+C\hbar\|\psi\|\|\mathfrak{P}_{\hbar,\theta}\psi\|\,.\]
Applying this inequality to functions in the space spanned by the $k$ first eigenfunctions\footnote{associated with eigenvalues repeated according to the multiplicity} of $\mathfrak{N}_\hbar^c$ (provided that $\lambda_k(\mathfrak{N}^c_{\hbar,\theta})\leq E$), we get
\[\langle {\rm Op}^w_\hbar\big(\overset{\circ}{\mu}_a(\cdot+\theta)\big)\mathfrak{P}_{\hbar,\theta}\psi,\mathfrak{P}_{\hbar,\theta}\psi\rangle\leq (1+\tilde C\hbar) \lambda_k(\mathfrak{N}^c_{\hbar,\theta})\|\psi\|^2+ C\hbar\|\psi\|\|\mathfrak{P}_{\hbar,\theta}\psi\|\,,\]
and also
\begin{equation}\label{eq.ineqPpsipsi}
\langle{\rm Op}^w_\hbar\big(\overset{\circ}{\mu}_a(\cdot+\theta)\big)\mathfrak{P}_{\hbar,\theta}\psi,\mathfrak{P}_{\hbar,\theta}\psi\rangle\leq  (\lambda_k(\mathfrak{N}^c_{\hbar,\theta})+C\hbar)\|\psi\|^2\,.
\end{equation}
We have now to check that, when $\psi$ runs over our $k$-dimensional space, $\mathfrak{P}_{\hbar,\theta}\psi$ runs over a $k$-dimensional space.  Using the first equality in \eqref{eq:cor.1-1} with the $j$-th eigenfunction $\psi=\psi_j$ and $z=\lambda_j(\mathfrak{N}_\hbar^c)$, and by using the Agmon estimates, we see that, there exists $C>0$ such that for all $j, \ell$, 
\[\left|\langle \mathfrak{P}^*_{\hbar,\theta}\mathfrak{P}_{\hbar,\theta}\psi_j,\psi_\ell\rangle-\delta_{j\ell}\right|\leq C\hbar\,.\]
Then, writing $\psi=\sum_{j=1}^k\alpha_j\psi_j$, we have
\begin{equation}\label{eq.Ppsipsi}
\left\|\mathfrak{P}_{\hbar,\theta}\psi\right\|^2=\Re\sum_{j,\ell=1}^k\alpha_j\overline{\alpha_\ell}\langle\mathfrak{P}_{\hbar,\theta}\psi_j,\mathfrak{P}_{\hbar,\theta}\psi_\ell\rangle\geq (1-C\hbar)\sum_{j=1}^k|\alpha_j|^2=(1-C\hbar)\|\psi\|^2\,.
\end{equation}
Recalling \eqref{eq.ineqPpsipsi} and using the min-max theorem, this shows that there exist $C,\hbar_0>0$ such that for all $\hbar\in(0,\hbar_0)$,
\[\lambda_k\left({\rm Op}^w_\hbar\big(\overset{\circ}{\mu}_a(\cdot+\theta)\big)\right)\leq \lambda_k(\mathfrak{N}^c_{\hbar,\theta})+C\hbar\,,\]
provided that $\lambda_k(\mathfrak{N}^c_{\hbar,\theta})\leq E$. By using Proposition \ref{prop.param} and similar arguments, we get the reversed inequality. Let us only sketch the proof. Thanks  to Proposition \ref{prop.param}, we get, for all $f\in L^2_{loc}(\mathbb{R})$ that is $2L$-periodic,
\[\Re\langle(\mathfrak{N}^c_{\hbar,\theta}-z)(Q^+_{\hbar,\theta}f),Q^+_{\hbar,\theta}f\rangle\leq -\Re\langle \mathfrak{P}^*_{\hbar,\theta}Q^\pm_{\hbar,\theta}(z)f,Q^+_{\hbar,\theta}f\rangle+C\hbar^3\|f\|\|Q^+_{\hbar,\theta}f\|\,.\]
By taking $z=0$ and by using the Calder\'on-Vaillancourt theorem to deal with the right-hand-side, we get
\[\langle\mathfrak{N}^c_{\hbar,\theta}(Q^+_{\hbar,\theta}f),Q^+_{\hbar,\theta}f\rangle\leq \Re\langle {\rm Op}^w_\hbar\overset{\circ}{\mu}_a(\cdot+\theta)f,\mathfrak{P}_{\hbar,\theta}Q^+_{\hbar,\theta}f\rangle+C\hbar\|f\|^2\,.\]
Then, we have
\[\langle\mathfrak{N}^c_{\hbar,\theta}(Q^+_{\hbar,\theta}f),Q^+_{\hbar,\theta}f\rangle\leq \Re\langle {\rm Op}^w_\hbar\overset{\circ}{\mu}_a(\cdot+\theta)f,f\rangle+\tilde C\hbar\|f\|^2\,.\]
We can check that $\|Q^+_{\hbar,\theta}f\|\geq c\|f\|$ for some $c>0$.
From the min-max theorem, we infer that
\[\lambda_k(\mathfrak{N}^c_{\hbar,\theta})\leq \lambda_k\left({\rm Op}^w_\hbar\big(\overset{\circ}{\mu}_a(\cdot+\theta)\big)\right)+C\hbar\,.\]
There exist $C,\hbar_0>0$ such that for all $k\geq 1$ and all $\hbar\in(0,\hbar_0)$,
\[\left|\lambda_k\left({\rm Op}^w_\hbar\big(\overset{\circ}{\mu}_a(\cdot+\theta)\big)\right)- \lambda_k(\mathfrak{N}^c_{\hbar,\theta})\right|\leq C\hbar\,,\]
as soon as $\lambda_k(\mathfrak{N}^c_{\hbar,\theta})\leq E$.

It remains to apply the usual Weyl estimate  available for a $\hbar$-pseudodifferential   operator whose principal symbol is $\overset{\circ}{\mu}_a(\sigma+\theta)$ and remember that $\theta\to0$ when $\hbar\to 0$ and that the symbol is $2L$-periodic with respect to $s$.
\end{proof}

\section{Estimate of the bottom of the spectrum}\label{sec:proofs}
Let us now focus on the bottom of the spectrum. Here, we follow the analysis in \cite[Section 8.3]{BLTRS21}, where quite similar considerations were used in the context of the magnetic Dirac operator. In this section, we only highlight the most important steps.
We will sometimes write $\sigma_a=\sigma(a)$ to lighten the notation in this section.

We consider Proposition \ref{prop.param} with $z\in[0,\beta_a+C\hbar]$. In view of \eqref{eq:param1}, this suggests to consider the operator whose Weyl symbol is
\begin{equation}\label{eq.exp.eff-symb}
p^{\mathrm{eff}}_\hbar(s,\sigma)=\overset{\circ}{\mu}_a(\sigma)-\hbar \hat q_{1}^\pm(s,\sigma)-\hbar^2\hat q_{2,\beta_a}^\pm(s,\sigma)\,.
\end{equation}
We let
\[
p^{\mathrm{eff}}_{\hbar,\theta}(s,\sigma)= p^{\mathrm{eff}}_\hbar(s,\sigma + \theta)\,.
\]
\begin{proposition}\label{prop.Nc-eff}
We have, for all $n\geq 1$,
\[\lambda_n(\mathfrak{N}^c_{\hbar,\theta})=\lambda_n\big(\mathrm{Op}^w_\hbar (p^{\mathrm{eff}}_{\hbar,\theta})\big)+o(\hbar^{2})\,,\]
uniformly with respect to $\theta\in\mathbb{R}$.	
\end{proposition}
\begin{proof}
Let us only sketch the proof. We recall that we have \eqref{eq:cor.1-1} and \eqref{eq:cor.1-2}. Thus, for all $\psi$ in the space spanned by the  $n$ first eigenfunctions associated with  the first $n$ eigenvalues of $\mathfrak{N}^c_{\hbar,\theta}$ (which all approach $\beta_a$, as we can check thanks to similar manipulations as in the proof of Proposition~\ref{cor.Weyl}),
\[\|Q_{\hbar,\theta}^\pm(z)\mathfrak{P}_{\hbar,\theta}\psi\|\leq C\|(\mathfrak{N}_{\hbar,\theta}^c-z)\psi\|+C\hbar^3\|\psi\|\,,\]
where we used the Agmon estimates to deal with the term of order $\hbar^3$. Applying this to $z$ such that $z=\beta_a+o(1)$, we see that
\[\|\big(\mathrm{Op}^w_\hbar (p^{\mathrm{eff}}_{\hbar,\theta})-z\big)\mathfrak{P}_{\hbar,\theta}\psi\|\leq C\|(\mathfrak{N}_{\hbar,\theta}^c-z)\psi\|+o(\hbar^2)\|\psi\|\,.\]
With \eqref{eq.Ppsipsi} and the Spectral Theorem\footnote{Use $z=\lambda_j((\mathfrak{N}_{\hbar,\theta}^c)$ and take $\psi$ in the corresponding eigenspace.}, this shows that the $n$ first eigenvalues (repeated with multiplicity)  of $\mathfrak{N}_{\hbar,\theta}^c$ lie at a distance $o(\hbar^2)$ to the spectrum of $\mathrm{Op}^w_\hbar (p^{\mathrm{eff}}_{\hbar,\theta})$. In particular, this gives the lower bound
\[\lambda_n(\mathfrak{N}^c_{\hbar,\theta})\geq\lambda_n\big(\mathrm{Op}^w_\hbar (p^{\mathrm{eff}}_{\hbar,\theta})\big)+o(\hbar^{2})\,.\]
The upper bound follows from similar arguments.
\end{proof}

Then, we can check that the eigenfunctions of $\mathrm{Op}^w_\hbar( p^{\mathrm{eff}}_{\hbar,\theta})$ are microlocalized with respect to $\sigma+\theta$ near $\sigma(a)$ at the scale $\hbar^{\frac{\gamma}{2}}$ (for all $\gamma\in(0,1)$) by using that the principal symbol has a unique minimum, which is non-degenerate. This leads us to write the Taylor expansion 
\begin{multline*}
p^{\mathrm{eff}}_{\hbar}(s,\sigma)=\frac{\mu''_a(\sigma_a)}{2}(\sigma-\sigma_a)^2-\hbar\hat q_{1}^\pm(s,\sigma_a)-\hbar(\sigma-\sigma_a)\partial_\sigma \hat q_1^\pm(s,\sigma_a)-\hbar^2q_{2,\beta_a}^\pm(s,\sigma_a)\\
+\mathscr{O}(\hbar(\sigma-\sigma_a)^2+\hbar^2(\sigma-\sigma_a)+(\sigma-\sigma_a)^3)\,.
\end{multline*}
Rearranging the first terms, we get
\begin{equation}\label{eq.aefftob}
	p^{\mathrm{eff}}_\hbar(s,\sigma)=b_\hbar(s,\sigma)
	+\mathscr{O}(\hbar(\sigma-\sigma_a)^2+\hbar^2(\sigma-\sigma_a)+(\sigma-\sigma_a)^3)\,,
\end{equation}
where
\begin{multline}\label{eq.bhbardiff}
b_\hbar(s,\sigma)=\frac{\mu''_a(\sigma_a)}{2}\left(\sigma-\sigma_a-\hbar\frac{\partial_\sigma \hat q_1^\pm(s,\sigma_a)}{\mu''_a(\sigma_a)}\right)^2-\hbar q_{1}^\pm(s,\sigma_a)\\
+\hbar^2\left(q_{2,\beta_a}^\pm(s,\sigma_a)+\frac{(\partial_\sigma  q_1^\pm(s,\sigma_a)^2}{2\mu''_a(\sigma_a)}\right)\,.
\end{multline}
We let
\[b_{\hbar,\theta}(s,\sigma)=b_\hbar(s,\sigma+\theta)\,.\]
Note that $\mathrm{Op}^w_\hbar b_{\hbar,\theta}$ is a differential operator of order $2$ and that it shares common features with that of \cite[(8.10)]{BLTRS21}. The difference is the presence of the a priori non-zero term $\hbar\,\hat q_{1}^\pm(s,\sigma_a)$. In Lemmas \ref{lem.q1=Ck} and \ref{lem.a=-1}, we describe the terms appearing in \eqref{eq.bhbardiff}.

\begin{lemma}\label{lem.q1=Ck}
When $a>-1$, 
	\[ q_{1}^\pm(s,\sigma_a)=C(a)k(s)+{ \mathscr{O}(\hbar^\infty)}\,,\]
	with $C(a)= -M_3(a)>0$, with $M_3(a)$ defined in \eqref{eq:moments} and calculated in \eqref{eq:m3*}. When $a=-1$, we have $ q_{1}^\pm(s,\sigma_a)=\mathscr{O}(\hbar^\infty)$. 
\end{lemma}
\begin{proof}
By \eqref{eq:param20} and the definition of $n_1$ in \eqref{eq.nj},
\[   q_{1}^\pm(s,\sigma_a)= -k(s)\int_\R c_\mu \big(2t(\sigma_a-b_a t)^2+b_a t^2(\sigma_a-b_a t)\big)|\phi_a(t)|^2\mathrm{d}t\,,\]
where $\phi_a=\overset{\circ}{u}_{\sigma_a}$  and where $c_{\mu}$ was defined in \eqref{eq:c-delta}.  Since $\phi_a$ decays exponentially at $\pm\infty$,  we get
\begin{align*}  q_{1}^\pm(s,\sigma_a)&= -k(s)\int_\R \big(2t(\sigma_a-b_a t)^2+b_a t^2(\sigma_a-b_a t)\big)|\phi_a(t)|^2\mathrm{d}t+\mathscr{O}(\hbar^\infty)\\
&=-k(s)M_3(a)+\mathscr{O}(\hbar^\infty)\,,
\end{align*}
where we used  \eqref{aqddid}.  By \eqref{eq:m3*},  $M_3(-1)=0$  and $M_3(a)<0$ for $-1<a<0$. 
\end{proof}

\begin{lemma}\label{lem.a=-1}
	When $a=-1$, we have
	\[q_{2,\beta_a}^\pm(s,\sigma(a))=C_0  k(s)^2  +\mathscr O(\hbar^{\infty}),\quad{\rm and}\quad \partial_\sigma  q_1^\pm(s,\sigma(a))=0\,,\]
	with $C_0<0$ a universal constant.
\end{lemma}

The proof below establishes that $C_0 = -\frac{1}{4} + G$, where $G$ is given by \eqref{eq:ValueG}.

\begin{proof}
Let us recall  \eqref{eq:param20} and \eqref{eq.nj}.
For $a=-1$ ,  the function $\tau\mapsto \overset{\circ}{u}_{\sigma}(\tau)$ is even and  the functions
	\[ 
	\begin{aligned}
		\tau\mapsto &~n_1(s,\sigma)=k(s)c_\mu(\tau)\Big(2\tau(\sigma-b_a\tau)^2+b_a\tau^2(\sigma-b_a\tau)\Big)\,,\\
		\tau\mapsto &~ \partial_\sigma n_1(s,\sigma)=k(s)c_\mu(\tau)\Big( 4\tau(\sigma-b_\tau)+b_a\tau^2\Big)
	\end{aligned}  \]
	are odd.   So we get $\partial_\sigma q^\pm_1(s,\sigma(a))|_{a=-1}=0$.
	
	By the  same considerations,  using \eqref{eq:param2} and \eqref{eq.n0,1,2},  we have
	\[
	q_{2,\beta_a}^\pm(s,\sigma(a))|_{a=-1}=-\frac{k(s)^2}{4}+k(s)^2G +\mathscr O(\hbar^{\infty})\,,
	\]
	where (recall the function $f_0$ defined in Section~\ref{sec.dG})
	\begin{equation}\label{eq.G}
		G=2\langle v,w\rangle -2\int_0^{+\infty} \Big(3t^2(\xi_0-t)^2+2t^3(\xi_0-t)+\frac14t^4\Big)|f_0(t)|^2\mathrm{d}t\,.
	\end{equation}
	Here $w=\big(2t(\xi_0-t)^2+t^2(\xi_0-t)\big)f_0(t)$ and  $v$ is  the  unique solution of
	\[\begin{cases}\label{eq.f=Rw}
		-v''+(\xi_0-t)^2v -\Theta_0  v=w&{\rm  on~}\R_+\,,\\
		v(0)=0\,.
	\end{cases}
	\]
We will prove by a somewhat lengthy but elementary calculation that 
\begin{align}\label{eq:ValueG}
G=-7M_4+\frac32\xi_0M_3+\frac32\Theta_0^2=-\frac{21}{8}-\frac{9}8\Theta_0^2-\frac{57}{4}\xi_0M_3\,.
\end{align}
Numerically (see \cite{Bon12}),  $\xi_0\approx\sqrt{0.59}$ and $M_3=-C_1/2$ with   $0.858\leq 3C_1\leq 0.888$. Consequently $G<0$. So to finish the proof of Lemma~\ref{lem.a=-1} it only remains to prove \eqref{eq:ValueG}.
	
	For all  $k\geq 1$, we set
	\[ P_k=(\xi_0-t)^k\]
	and we observe that
	\[
	\begin{aligned}
		3t^2(\xi_0-t)^2&=3P_4-6\xi_0P_3+3\Theta_0P_2\,,\\
		2t^3(\xi_0-t)&=-2P_4+6\xi_0P_3-6\Theta_0P_2+2\xi_0\Theta_0P_1\,,\\
		\frac14t^4&=\frac14P_4-\xi_0P_3+\frac32\Theta_0P_2-\xi_0\Theta_0P_1+\Theta_0^2\,.
	\end{aligned}
	\]
	Consequently,
	\begin{multline*}
		2\int_0^{+\infty} \Big(3t^2(\xi_0-t)^2+2t^3(\xi_0-t)+\frac14t^4\Big)|f_0(t)|^2dt\\
		=\frac{5}2M_4-2\xi_0M_3-3\Theta_0M_2+2\xi_0\Theta_0M_1+2\Theta_0^2\\
		=\frac52M_4-2\xi_0M_3+\frac{\Theta_0^2}2\,.
	\end{multline*}
	Let us now compute 
	\begin{equation}\label{eq:vw=Pvf0}
		\langle v,w\rangle=\langle Pv,f_0\rangle
	\end{equation}
	where
	\[P(t)=2t(\xi_0-t)^2+t^2(\xi_0-t)=-(\xi_0-t)^3+\Theta_0(\xi_0-t)=-P_3(t)+\Theta_0P_1(t)\,. \]
	Let $p,q$ be two polynomial functions such that
	\[ v_0:=pf_0+qf_0'\]
	satisfies
	\[ \begin{cases}
		-v_0''+(\xi_0-t)^2v_0-\Theta_0v_0=Pf_0~{\rm  on~}\R_+\,,\\
		v_0(0)=0\,,
	\end{cases} \]
	thereby yielding the condition $p(0)=0$ and
	\[ \big(-p''+2P_1q+(\Theta_0-P_2)q' \big)f_0+( -2p'-q'')f_0'=Pf_0~{\rm on~}\R_+\,.\]
	 We look for $p$ and $q$ satisfying the  condition $-2p'-q''=0$ and  $q$ in the form $q=aP_2+bP_1+c$,   where $a,b$ and $c$ are to be determined. \\
	  We find after  straightforward computations:
	\[\begin{aligned}
		-p''+2P_1q+&(\Theta_0-P_2)q'=-P_3+\Theta_0P_1\\
		&\Longleftrightarrow 4aP_3+3bP_2+2(c-a\Theta_0)P_1-b\Theta_0=-P_3+\Theta_0P_1\\
		&\Longleftrightarrow a=-\frac14,~b=0,~c=\frac{\Theta_0}4\,,
	\end{aligned} \]
	and therefore
	\[\begin{aligned}
		p(t)&=\frac{t}4, \\ ~q(t)&=-\frac14(\xi_0-t)^2+\frac{\Theta_0}4,\\ ~v=v_0&=pf_0+qf_0'
		=\frac14\big((\xi_0-P_1)f_0+(-P_2+\Theta_0)f_0' \big)\,.
	\end{aligned}\]
We can now compute \eqref{eq:vw=Pvf0}. Noticing that
	\[Pp=\frac14(P_4-\xi_0P_3-\Theta_0P_2+\xi_0\Theta_0P_1),  ~Pq=\frac14( P_5-\Theta_0P_3-\Theta_0P_2+\Theta_0^2 )\,, \]
	we have
	\[
	\langle Pv,f_0\rangle=\frac14(M_4-\xi_0M_3-\Theta_0M_2)
	-\frac14 \langle  ( P_5-\Theta_0P_3-\Theta_0P_2+\Theta_0^2 )f_0',f_0\rangle\,.
	\]
	After an  integration by parts, we have
	\[
	\begin{aligned}
		&2 \langle  ( P_5-\Theta_0P_3-\Theta_0P_2+\Theta_0^2 )f_0',f_0\rangle\\
		&=-2 \langle  f_0,( P_5-\Theta_0P_3-\Theta_0P_2+\Theta_0^2 )'f_0,f_0\rangle -|f_0(0)|^2( P_5-\Theta_0P_3-\Theta_0P_2+\Theta_0^2 )(0)\rangle\\
		&=-2 \langle  f_0,( -5P_4+3\Theta_0P_2+2\Theta_0P_1 )f_0,f_0\rangle+0\\
		&=10M_4-6\Theta_0M_2\,.
	\end{aligned}\]
	Therefore,
	\[ \langle Pv,f_0\rangle=\Big( -\frac52+\frac14\Big)M_4-\frac{\xi_0}4M_3+\frac54\Theta_0M_2\,. \]
	Inserting this into \eqref{eq:vw=Pvf0}, we infer from \eqref{eq.G} that \eqref{eq:ValueG} is true.
	This finishes the proof.
\end{proof}

The study of the differential operator $\mathrm{Op}^w_\hbar ( b_{\hbar,\theta})$ is rather easy and the behavior of the spectrum depends on $a$.

When $a>-1$, thanks to our assumption on the maximum of the curvature, we are reduced to use the harmonic approximation  at $(s_{\max},\sigma(a))$ and we get the following.

\begin{proposition}[Case $a>-1$]\label{prop:finala>}
	When $a>-1$ and $k$ has a unique maximum which is non-degenerate, we have
	\[ \lambda_n\big(\mathrm{Op}^w_\hbar (b_{\hbar,\theta})\big)=-C(a)k_{\max}\hbar+(n-1/2)\hbar^{\frac32}\sqrt{-C(a)\mu''_a(\sigma(a))k''(s_{\max}) }+o(\hbar^{\frac32})\,,\]
	uniformly with respect to $\theta\in\mathbb{R}$.
\end{proposition}
In the  case  $a=-1$, there is essentially nothing to do. 

\begin{proposition}[Case $a=-1$]\label{pro:final-1}
	When $a=-1$, we have
	\[ \lambda_n\big(\mathrm{Op}^w_\hbar (b_{\hbar,\theta})\big)=\hbar^2\lambda_n\left(\mathcal{B}_{\hbar,\theta}\right)+\mathscr{O}(\hbar^\infty)\,,\]
	where
	\begin{equation*}
		\mathcal{B}_{\hbar,\theta}=\frac{\mu''_a(\sigma(a))}{2}\left(D_s+\hbar^{-1}\theta-\hbar^{-1}\sigma(a)\right)^2\\
		+C_0  k(s)^2\,.
	\end{equation*}
\end{proposition}

Taking $\theta=\theta(\hbar)$ (see \eqref{eq:defthetaa}) and arguing as in \cite[Section 8.3]{BLTRS21} to deal with the remainders in \eqref{eq.aefftob}, we deduce Theorems \ref{corol.a>-1} and \ref{corol.a=-1} from Propositions \ref{prop.coincide} and  \ref{prop.Nc-eff}.
Since there has been a number of changes of notation along the way, let us guide the reader to this conclusion.
Recall that $\hbar=h^{1/2}$. To prove Theorem~\ref{corol.a>-1}, by Proposition~\ref{prop:BHR2.7-bis} it suffices to prove the eigenvalue asymptotics for $\lambda_n(\mathfrak N_{\hbar,\theta})$.
By Proposition~\ref{prop.coincide} it suffices to consider the operator $\mathfrak{N}^c_{\hbar,\theta}$ (defined just before the proposition), and by Proposition~\ref{prop.Nc-eff} to consider $\mathrm{Op}^w_\hbar (p^{\mathrm{eff}}_{\hbar,\theta})$, which by \eqref{eq.aefftob} and the localization estimates reduces to the statement of Proposition~\ref{prop:finala>}.
The proof of Theorem~\ref{corol.a=-1} follows the same lines, only applying Proposition~\ref{pro:final-1} in the last step instead of Proposition~\ref{prop:finala>}.

\bibliographystyle{abbrv}
\bibliography{biblio.bib}

\begin{thebibliography}{10}

\bibitem{AHK}
W.~Assaad, B.~Helffer, and A.~Kachmar.
\newblock Semi-classical eigenvalue estimates under magnetic steps.
\newblock {\em arXiv:2108.03964}, 2021.

\bibitem{AK20}
W.~Assaad and A.~Kachmar.
\newblock Lowest energy band function for magnetic steps.
\newblock {\em J. Spectral Theory, arXiv:2012.13794}, 2020.

\bibitem{BLTRS21}
J.-M. Barbaroux, L.~Le~Treust, N.~Raymond, and E.~Stockmeyer.
\newblock {On the {D}irac bag model in strong magnetic fields}.
\newblock {\em arXiv:2007.03242}, 2020.

\bibitem{BN}
V.~Bonnaillie.
\newblock On the fundamental state energy for a {S}chr\"{o}dinger operator with
  magnetic field in domains with corners.
\newblock {\em Asymptot. Anal.}, 41(3-4):215--258, 2005.

\bibitem{Bon12}
V.~Bonnaillie-No{\"e}l.
\newblock Harmonic oscillators with {Neumann} condition on the half-line.
\newblock {\em Commun. Pure Appl. Anal.}, 11(6):2221--2237, 2012.

\bibitem{BND06}
V.~Bonnaillie-No\"{e}l and M.~Dauge.
\newblock Asymptotics for the low-lying eigenstates of the {S}chr\"{o}dinger
  operator with magnetic field near corners.
\newblock {\em Ann. Henri Poincar\'{e}}, 7(5):899--931, 2006.

\bibitem{BHR22}
V.~Bonnaillie-No\"{e}l, F.~H\'{e}rau, and N.~Raymond.
\newblock Purely magnetic tunneling effect in two dimensions.
\newblock {\em Invent. Math.}, 227(2):745--793, 2022.

\bibitem{FLTRVN22}
R.~Fahs, L.~Le~Treust, N.~Raymond, and S.~V\~{u}~Ng\d{o}c.
\newblock {Edge states for the Robin magnetic Laplacian}.
\newblock Manuscript in preparation, 2022.

\bibitem{FH06}
S.~Fournais and B.~Helffer.
\newblock Accurate eigenvalue asymptotics for the magnetic {N}eumann
  {L}aplacian.
\newblock {\em Ann. Inst. Fourier (Grenoble)}, 56(1):1--67, 2006.

\bibitem{FH10}
S.~Fournais and B.~Helffer.
\newblock {\em Spectral methods in surface superconductivity}, volume~77 of
  {\em Progress in Nonlinear Differential Equations and their Applications}.
\newblock Birkh\"{a}user Boston, Inc., Boston, MA, 2010.

\bibitem{FHK22}
S.~Fournais, B.~Helffer, and A.~Kachmar.
\newblock Tunneling effect induced by a curved magnetic edge, 2022.
\newblock Manuscript in preparation.

\bibitem{FK}
S.~Fournais and A.~Kachmar.
\newblock On the energy of bound states for magnetic {Schr{\"o}dinger}
  operators.
\newblock {\em J. Lond. Math. Soc., II. Ser.}, 80(1):233--255, 2009.

\bibitem{F07}
R.~L. {Frank}.
\newblock {On the asymptotic number of edge states for magnetic Schr\"odinger
  operators}.
\newblock {\em {Proc. Lond. Math. Soc. (3)}}, 95(1):1--19, 2007.

\bibitem{GMS91}
C.~G\'{e}rard, A.~Martinez, and J.~Sj\"{o}strand.
\newblock A mathematical approach to the effective {H}amiltonian in perturbed
  periodic problems.
\newblock {\em Comm. Math. Phys.}, 142(2):217--244, 1991.

\bibitem{GV}
A.~Giunti and J.~J.~L. Vel{\'a}zquez.
\newblock {Edge States for generalised Iwatsuka models: Magnetic fields having
  a fast transition across a curve}.
\newblock {\em arXiv:2109.09651}, 2021.

\bibitem{HM01}
B.~Helffer and A.~Morame.
\newblock Magnetic bottles in connection with superconductivity.
\newblock {\em J. Funct. Anal.}, 185(2):604--680, 2001.

\bibitem{HeSj}
B.~Helffer and J.~Sj{\"o}strand.
\newblock Multiple wells in the semi-classical limit. {I}.
\newblock {\em Commun. Partial Differ. Equations}, 9:337--408, 1984.

\bibitem{HeSj5}
B.~Helffer and J.~Sj{\"o}strand.
\newblock Puits multiples en m{\'e}canique semi-classique. {V}: {\'E}tude des
  minipuits. ({Multiple} wells in semi-classical mechanics. {V}: {Study} of
  miniwells).
\newblock Current topics in partial differential equations, {Pap}. dedic. {S}.
  {Mizohata} {Occas}. 60th {Birthday}, 133-186 (1986)., 1986.

\bibitem{HS6}
B.~Helffer and J.~Sj\"{o}strand.
\newblock Puits multiples en m\'{e}canique semi-classique. {VI}. {C}as des
  puits sous-vari\'{e}t\'{e}s.
\newblock {\em Ann. Inst. H. Poincar\'{e} Phys. Th\'{e}or.}, 46(4):353--372,
  1987.

\bibitem{HPRS16}
P.~D. {Hislop}, N.~{Popoff}, N.~{Raymond}, and M.~P. {Sundqvist}.
\newblock {Band functions in the presence of magnetic steps}.
\newblock {\em {Math. Models Methods Appl. Sci.}}, 26(1):161--184, 2016.

\bibitem{KK}
A.~Kachmar and A.~Khochman.
\newblock Spectral asymptotics for magnetic {Schr{\"o}dinger} operators in
  domains with corners.
\newblock {\em J. Spectr. Theory}, 3(4):553--574, 2013.

\bibitem{Keraval}
P.~Keraval.
\newblock {\em Formules de Weyl par réduction de dimension. Applications à
  des Laplaciens électro-magnétiques}.
\newblock PhD thesis, Université de Rennes 1, 2018.

\bibitem{M07}
A.~Martinez.
\newblock A general effective {H}amiltonian method.
\newblock {\em Atti Accad. Naz. Lincei Rend. Lincei Mat. Appl.},
  18(3):269--277, 2007.

\bibitem{Ray}
N.~Raymond.
\newblock {\em Bound states of the magnetic {S}chr\"{o}dinger operator},
  volume~27 of {\em EMS Tracts in Mathematics}.
\newblock European Mathematical Society (EMS), Z\"{u}rich, 2017.

\bibitem{Simon}
B.~Simon.
\newblock The bound state of weakly coupled {S}chr\"{o}dinger operators in one
  and two dimensions.
\newblock {\em Ann. Physics}, 97(2):279--288, 1976.

\bibitem{Zworski}
M.~Zworski.
\newblock {\em Semiclassical analysis}, volume 138 of {\em Graduate Studies in
  Mathematics}.
\newblock American Mathematical Society, Providence, RI, 2012.

\end{thebibliography}

\end{document}